\documentclass[twoside,11pt,a4paper]{amsart}
\usepackage[utf8]{inputenc}

\usepackage{amsfonts,amsmath,amssymb,amsthm}
\usepackage{color}
\usepackage{graphicx}
\usepackage{latexsym}

\newcommand{\red}{\color{red}}

\usepackage[colorlinks,
linkcolor=red,
citecolor=blue
]{hyperref}

\newcommand{\R}{\mathbb{R}}

\newcommand{\dive}{\operatorname{div}}
\newcommand{\grad}{\nabla}

\def\eps{\varepsilon}

\def\pt{\partial_t}

\newtheorem{theorem}{Theorem}[section]
\newtheorem{remark}{Remark}[section]

\newtheorem{lemma}{Lemma}[section]

\numberwithin{equation}{section}

\begin{document}

\title[Large solutions for the Euler--Poisson system]{\bf 
Global relaxation limit for the one--fluid Euler--Poisson system with large smooth data}

\author{Yue-Jun Peng, Ling-Yun Shou, and Jiang Xu}
\date{}

\keywords{Euler--Poisson system; global classical solutions; large data; relaxation limit; large-time behavior}

\subjclass[2020]{35Q35, 76N10, 35B40, 35B65}

\begin{abstract}
Whether the multi-dimensional Euler--Poisson system admits global smooth solutions remains a  challenging open problem. In this paper, we construct a class of large--data global smooth solutions to the one--fluid Euler--Poisson system in $\mathbb{R}^d$ ($1\le d\le 5$) by using the relaxation dissipation mechanism.
Precisely, assuming that the initial density is far from vacuum and $\varepsilon E_0$ is sufficiently small, where $E_0$ denotes the initial energy and $\varepsilon$ is the relaxation time, we establish the global well--posedness of smooth solutions to the Cauchy problem. In particular, the size of the initial perturbation may be arbitrarily large, provided that the relaxation time is sufficiently small. Furthermore, we introduce an effective
unknown motivated by Darcy's law to derive quantitative error estimates at the rate $\mathcal O(e^{-\lambda t}\eps)$ between the rescaled Euler--Poisson system and the limiting  drift--diffusion system for ill--prepared data. The new ingredient lies in  developing the maximum principle for the nonlinear drift--diffusion system
with nonlocal effect, which leads to the large--data global existence.

\end{abstract}
\maketitle
\section{Introduction}
\vspace{2mm}

Plasma is widely present in nature and accounts for most of the observable matter in the universe, from stars and intergalactic space to nebulae and neon signs. In the classical two-fluid model describing plasma dynamics, the self-consistent electrostatic Newtonian potential satisfies the Poisson equation for the repulsive Coulomb interaction. When one species is regarded as a fixed neutralizing background, the system reduces to the one-fluid Euler--Poisson model \cite{c1,J}. It also arises in semiconductor theory to describe the macroscopic transport of charge carriers driven by an electric field \cite{MR1}.

In this paper, we consider the multi-dimensional one-fluid {\emph{Euler--Poisson system}} with friction in several space dimensions:
\begin{equation}\label{Euler}
\left\{
\begin{aligned}
  &\pt\rho+\dive q=0,\\
  &\pt q+\dive\Big(\dfrac{q\otimes q}{\rho}\Big)+\grad p(\rho)
=-\rho\grad\phi-\dfrac{q}{\eps},\\
  &-\Delta\phi=\rho-\bar{\rho},
\end{aligned}\right.
\end{equation}
with the initial condition
\begin{equation}\label{Eulerd}
(\rho, q)(0,x)=\bigl(\rho_0,q_0\bigr)(x),
\end{equation}
for $t>0$ and $x=(x_1,x_2,\cdots,x_d)\in\R^d$.
Here, $\eps\in(0,1]$ is the relaxation time, $\rho=\rho(t,x)>0$ and $q=q(t,x)\in\mathbb R^d$ denote the density and momentum of the charged particles, respectively, $p=p(\rho)$ is the pressure, and $\phi=\phi(t,x)$ is the electrostatic potential generated by the deviation of the charge density from a constant background state $\bar{\rho}>0$. Without loss of generality, we normalize the background density $\bar{\rho}=1$ throughout this paper.

In the following, we recall analytical results in the literature on the Euler--Poisson system,
which can be divided into two directions.

\vspace{3mm}

\noindent
{\textbf{Well-posedness.}} The Euler--Poisson system belongs to the general class of hyperbolic conservation laws, for which the local well-posedness theory of classical solutions has been well established; see \cite{Dafermos1,11}. On the other hand, for large amplitudes, classical solutions to the Euler--Poisson system without dissipation generally develop singularities in finite time; refer to, for instance, \cite{guoblow,7,MP,P,16} and the references therein.  In one dimension, the critical threshold phenomenon for global regularity versus finite-time blow-up was studied by Tadmor and Wei \cite{TW1} (see also Wei, Tadmor and Bae \cite{WTB}). We also refer to the existence of global weak solutions in one dimension \cite{prv} and high dimensions with spherical symmetry data \cite{CHWY}.

In contrast to the large-data results described above, the situation is markedly different for small perturbations of the constant equilibrium state $(1,0)$. In the presence of the velocity relaxation term, \eqref{Euler}  can be viewed as a partially dissipative system. Owing to the strict monotonicity of the pressure function $p$, the partial dissipation can be transferred to the density variable through the Euler--Poisson coupling. Combined with classical energy estimates, these mechanisms lead to the global existence of smooth solutions for sufficiently small perturbations; see, for example, \cite{A1,HMW,LMM}.

\vspace{3mm}

\noindent
{\textbf{Relaxation limit.}} Relaxation phenomena occur in a wide variety of physical situations, such as modeling blood flow with friction forces, non-equilibrium gas dynamics, kinetic theory, traffic flows, and more (see \cite{Dafermos1,W1}). These phenomena describe situations in which a physical system is perturbed away from a stable equilibrium state. In this case, systems are described by a set of equations
where the source terms contain time-relaxation parameters. We refer to \cite{liu1,N1} and the references therein for the study of relaxation limits from first-order hyperbolic systems to hyperbolic equilibrium equations. However, if a
slow time scaling is used, the solutions are expected to relax toward diffusion equilibrium.

Relaxation limits for damped Euler systems have attracted sustained attention over the past decades.  The results in one spatial dimension were investigated by  Marcati and Milani~\cite{12} and Marcati and Rubino~\cite{M1}. For the isothermal case, Junca and Rascle~\cite{5} established convergence to the heat equation for large $BV$ data away from vacuum.  In several dimensions, the weak relaxation limit for the damped Euler system to the porous medium equation has been proved in \cite{1,10,13,XuWang}. Concerning the explicit convergence rates, we refer to Li, Peng and Zhao \cite{9} in one dimension. 
In higher dimensions, global convergence rates in the ill-prepared setting were established in Sobolev spaces by Crin-Barat, Peng and Shou \cite{CBPS}, and in critical spaces by Crin-Barat and Danchin \cite{c3}. Finally, we also mention the work of Peng \cite{pengJFA} on this limit for large smooth solutions of the isothermal Euler equations.

For strong solutions, we analyze the diffusion limit for \eqref{Euler} under the scaling
\begin{align}\label{scaling}
\bigl(\rho^\eps,q^\eps,\phi^\eps\bigr)(t,x)
:=\left(\rho,\dfrac{1}{\eps}q,\phi\right)\left(\dfrac{t}{\eps},x\right).
\end{align}
Then the diffusively rescaled Euler--Poisson system, together with the
initial condition \eqref{Eulerd}, becomes
\begin{equation}\label{F1.2}
\left\{
\begin{aligned}
  &\pt\rho^\eps+\dive q^\eps=0,\\
  &\eps^2 \pt q^\eps+\eps^2\dive \Big(\dfrac{q^\eps\otimes q^\eps}{\rho^\eps} \Big)
+\grad p(\rho^\eps)=-\rho^\eps\grad\phi^\eps-q^\eps,\\
  &-\Delta\phi^\eps=\rho^\eps-1,\\
  &(\rho^\eps,q^\eps)(0,x)=\left( \rho_0,\frac{q_0}{\eps}\right)(x).
\end{aligned}\right.
\end{equation}

Formally, we denote
\[
\bigl(\rho^*,q^*,\phi^*\bigr)
:=\lim_{\eps\to 0}\bigl(\rho^\eps,q^\eps,\phi^\eps\bigr),
\]
and set $\rho_0^*=\rho_0$. As $\eps\to 0$, we expect that 
\begin{equation}\label{F1.5}
\left\{
\begin{aligned}
   &\pt\rho^*+\dive q^*=0,\\
   &\grad p(\rho^*)=-\rho^*\grad\phi^*-q^*\,\,(\text{Darcy's law}),\\
   &-\Delta\phi^*=\rho^*-1,
\end{aligned}\right.
\end{equation}
which yields a {\emph{drift--diffusion system}} together with an initial condition:
\begin{equation}\label{F1.6}
\left\{
\begin{aligned}
  &\pt\rho^*-\Delta p(\rho^*)=\dive(\rho^*\grad\phi^*),\\
  &-\Delta\phi^*=\rho^*-1,\\
  &\rho^*(0,x)=\rho_0^*(x).
\end{aligned}\right.
\end{equation}
System \eqref{F1.6} can be decoupled into the equation of $\rho^*$:
\begin{equation}\label{DDp}
\left\{
\begin{aligned}
&\pt\rho^*-\Delta p(\rho^*)=\dive\bigl(\rho^*\grad(-\Delta)^{-1}(\rho^*-1)\bigr),\\
&\rho^*(0,x)=\rho_0^*(x),
\end{aligned}
\right.
\end{equation}
and the equation of the electric field:
\begin{align}\label{nablaPhi}
\grad\phi^*=\grad(-\Delta)^{-1}(\rho^*-1).
\end{align}

The zero-relaxation limit of Euler--Poisson type systems toward drift--diffusion models has been extensively studied in the literature. Early contributions in one dimension include the works of Marcati and Natalini \cite{MN}, Lattanzio and Marcati \cite{LM1}, J\"ungel and Peng \cite{JP1,JP2} and Al\'i, Bini and Ri \cite{ABR1}. Relative entropy methods have recently been developed to derive the convergence rates on finite time intervals; see, for instance, Lattanzio and Tzavaras \cite{LT1}. For classical solutions, the local-in-time convergence and initial layer analysis of classical solutions have been justified by Yong \cite{Y1} and Hajjej and Peng \cite{HP}. In the perturbative regime near constant equilibria, Xu \cite{XuSIAM} and Peng \cite{PEP1} independently proved uniformly global classical solutions together with global weak convergence of the relaxation limit for Euler--Poisson systems. More recently, Li, Peng and Zhao \cite{LPZ1} derived global strong convergence for smooth  solutions in periodic domains with explicit rates. 

\vspace{3mm}

\noindent
{\textbf{Our contributions.}} To the best of our knowledge, in several space dimensions all available results on global classical solutions and the relaxation limit for the Euler--Poisson system are restricted to the setting of small initial perturbations. 
\emph{It remains a challenging problem whether the Euler--Poisson system admits global well-posedness and the corresponding relaxation limit toward the drift--diffusion equation for large initial data.}

Our main goal is to establish the global existence of smooth solutions to the Euler--Poisson system for a class of large initial data. We focus on the isothermal pressure law:
\begin{equation}\label{F1.1}
p(\rho)=a^2\rho,
\end{equation}
where $a>0$ denotes the sound speed. 
This choice is motivated by the classical work of Nishida and Smoller \cite{Ni68, NS} on the one-dimensional Cauchy problem for the isentropic Euler equations and the recent progress \cite{pengJFA} on the relaxation limit from the Euler equations to the linear heat equation. Compared with the Euler equations without the Poisson equation, the limiting  system \eqref{F1.6} or 
\eqref{DDp}-\eqref{nablaPhi} remains {\emph{nonlinear and nonlocal}}. This feature requires new techniques to handle the difficulties caused by the nonlinear effects. On the other hand, our work further reveals the nonlinear stabilizing mechanisms generated by the Poisson coupling in the relaxation limit. 

More precisely, we define the initial electric field
$\nabla\phi_0:=\nabla(-\Delta)^{-1}(\rho_0-1)$ and the initial energy associated with \eqref{Euler}--\eqref{Eulerd} by
\begin{align}
E_0:=\|\rho_0-1\|_{H^m}+\|q_0\|_{H^m}+\|\nabla\phi_0\|_{L^2},\quad H^m:=H^m(\R^d).
\end{align}
We can prove that the Cauchy problem \eqref{Euler}--\eqref{Eulerd} admits a unique global classical solution, provided that $E_0\eps$ is bounded by a sufficiently small constant independent of $\eps$. Indeed, that condition covers both the case of small perturbations where $E_0$ itself is sufficiently small, and the case of large data where $E_0$ may be arbitrarily large as long as the relaxation parameter $\eps$ is sufficiently small. 



\begin{theorem}\label{theorem1.1}
Let $1\le d\le 5$, $m>\frac{d}{2}+1$ be an integer and \eqref{F1.1} hold. Assume that $(\rho_0, q_0,\nabla\phi_0)$ with $\nabla\phi_0:=\nabla(-\Delta)^{-1}(\rho_0-1)$ fulfills
\begin{align}\label{1.11}
   & (\rho_0-1, q_0)\in H^m,\quad \nabla\phi_0\in L^2,\quad \rho_1\leq \rho_0(x)\leq \rho_2,\quad x\in \mathbb{R}^d,
\end{align}
for two positive constants $\rho_1\leq \rho_2$.  There exists a positive constant $\delta_0$ depending only on $\rho_1, \rho_2$, $d$, $a$ and $E_0$ such that if
\begin{equation}
\begin{aligned}
E_0\eps \leq \delta_0,\label{1.12}
  \end{aligned}
\end{equation}  
then the Euler--Poisson system \eqref{Euler}--\eqref{Eulerd} has a unique global classical solution $(\rho,q,\phi)$ satisfying
\begin{align}
&(\rho-1,q)\in C(\mathbb{R}_+;H^{m})\cap L^2(\mathbb{R}_+;H^{m}),\quad \nabla\phi\in C(\mathbb{R}_+;L^2)\cap L^2(\mathbb{R}_+;L^2),\\
&0<\rho_1-C_0 E_0 \eps\leq \rho(t,x)\leq \rho_2+C_0 E_0 \eps,\quad (t,x)\in \mathbb{R}_+\times\mathbb{R}^d,\label{lowupper0}
\end{align}
and the exponential stability
\begin{equation}\label{exp:EP}
\begin{aligned}
&\quad \sup_{t\in\mathbb{R}_+}e^{\lambda_0 \eps t} \big(\|\rho(t)-1\|_{H^m}^2+\|q(t)\|_{H^m}^2+\|\nabla\phi(t)\|_{L^{2}}^2\big)\\
&+\int_0^\infty e^{\lambda_0 \eps t} \big( \eps \|\rho(t)-1\|_{H^m}^2+\frac{1}{\eps}\|q(t)\|_{H^m}^2+\eps \|\nabla\phi(t)\|_{L^2}^2\big)\,dt\leq C_0 E_0^2.
\end{aligned}
\end{equation}
Here, $C_0$ and $\lambda_0 >0$ are constants independent of $\eps$. 
\end{theorem}

\begin{remark}
Some comments on Theorem \ref{theorem1.1} are in order.
\begin{itemize}
\item Condition \eqref{1.12} allows the assumption that $E_0\ll 1$, which recovers the small-perturbation global well-posedness result, see for example \cite{PEP1, XuSIAM}. More importantly, \eqref{1.12}  remains valid for arbitrarily large initial data provided that the relaxation parameter $\eps$ is sufficiently small, which is a surprising result for the Euler--Poisson system \eqref{Euler}--\eqref{Eulerd}. 


\item If $(\rho^\eps, q^\eps, \phi^\eps)$
is given by the scaling \eqref{scaling}, then we
get the following uniform exponential decay estimate:
\begin{equation}\label{exp:EP1}
\begin{aligned}
&\quad \sup_{t\in \mathbb{R}_+} e^{\lambda_0 t} \big(\|\rho^\eps(t)-1\|_{H^m}^2+\eps^2\|q^\eps(t)\|_{H^m}^2+\|\nabla\phi^\eps(t)\|_{L^{2}}^2\big)\\
&+\int_0^\infty e^{\lambda_0 t} \big( \|\rho^\eps(t)-1\|_{H^m}^2+\|q^\eps(t)\|_{H^m}^2+\|\nabla\phi^\eps(t)\|_{L^2}^2\big)\,dt\leq  C_0 E_0^2,
\end{aligned}
\end{equation}
which is different from the usual situation for partially dissipative hyperbolic systems in the whole space, where the classical Shizuta--Kawashima theory only yields algebraic decay rates, cf. \cite{SK}.
In fact, the Poisson coupling has an enhanced stabilizing effect, which leads to the exponential decay to equilibrium of classical solutions.


\item The regularity assumption on $\nabla\phi_0$ has been weakened in comparison with previous results. Indeed, it follows from the standard Fourier multiplier theorem (see \cite{stein}) that  $\|\nabla(-\Delta)^{-1}f\|_{\dot H^{s+1}}\sim \|f\|_{\dot H^s}$  for $s\in\mathbb R$. Consequently, $\rho_0-1\in H^m$ and $\nabla\phi_0\in L^2$ in \eqref{1.11} are equivalent to $\nabla\phi_0\in H^{m+1}$.

\item For large initial data, the results in Theorem \ref{theorem1.1} are based on the global well-posedness of classical solutions for the drift-diffusion system \eqref{F1.6} and \eqref{F1.1}, which is obtained under the restriction condition on the space dimension $d\le 5$. So far, we don't know whether this condition could be removed in Theorem \ref{theorem1.1}.

\end{itemize}
\end{remark}


Our second result concerns asymptotic stability of those global large solutions to the Euler--Poisson system \eqref{Euler}--\eqref{Eulerd}. For ill-prepared initial data, we rigorously justify the relaxation limit from \eqref{Euler}--\eqref{Eulerd} to \eqref{F1.6} and obtain the improved convergence rate $\eps e^{-\lambda t}$.



\begin{theorem}\label{theorem1.3}
Let $(\rho,q,\phi)$ and $(\rho^*,\phi^*)$ be the solutions to the Euler--Poisson system \eqref{Euler}--\eqref{Eulerd} and to the drift--diffusion system \eqref{F1.6} addressed in
Theorems \ref{theorem1.1} and \ref{theorem1.2}, respectively, and $q^*$ be given by
Darcy's law $q^*=-\nabla p(\rho^*)-\rho^*\nabla\phi^*$. Let $(\rho^\eps, q^\eps, \phi^\eps)$
be given by the scaling \eqref{scaling}. 
Then we have the following error estimate
 \begin{equation}
 \begin{aligned}\label{error2}
 & \sup_{t\in\mathbb{R}_+}e^{\lambda_1 t}\|(\rho^\eps-\rho^*)(t)\|_{H^{m-1}}^2+\int_0^\infty e^{\lambda_1 t}\|(\rho^\eps-\rho^*)(t)\|_{H^{m}}^2\,dt\\
 &\quad+\sup_{t\in\mathbb{R}_+}e^{\lambda_1 t}\|\nabla(\phi^\eps-\phi^*)(t)\|_{L^2}^2+\int_0^\infty e^{\lambda_1 t}\|\nabla(\phi^\eps-\phi^*)(t)\|_{L^2}^2\,dt\\
 &\quad+\int_0^\infty e^{\lambda_1 t} \|(q^\eps-q^*-q_L^\eps)(t)\|_{H^{m-1}}^2\,dt\\
 &\quad+\int_0^\infty e^{\lambda_1 t} \|(q^\eps+\nabla p(\rho^\eps)+\rho^\eps\nabla\phi^\eps-Z_L^\eps)(t)\|_{H^{m-1}}^2\,dt\leq \tilde C E_0^2\eps^2,
 \end{aligned}
 \end{equation}
where $\tilde C$ and $\lambda_1$ are two positive constants independent of $\eps$, and the initial layer corrections $q_L^\eps$ and $Z_L^\eps$  respectively satisfy
\begin{equation}\label{qL}
\left\{
\begin{aligned}
&\eps^2\partial_t q_L^\eps+q_L^\eps=0,\\
&q_L^\eps(0,x)=\frac{1}{\eps}q_0(x),
\end{aligned}
\right.
\quad\text{and}\quad
\left\{
\begin{aligned}
&\eps^2\partial_t Z_L^\eps+Z_L^\eps=0,\\
&Z_L^\eps(0,x)=\frac{1}{\eps}q_0(x)+\nabla p(\rho_0)(x)+\rho_0\nabla\phi_0(x).
\end{aligned}
\right.
\end{equation}
Consequently, the following global strong convergence properties hold as $\eps\rightarrow0$:
\begin{alignat}{2}
\rho^\eps &\to \rho^*
&\quad& \text{strongly\,\, in\,\, } C(\mathbb{R}_+;H^{m-1})\cap L^2(\mathbb{R}_+;H^m), \label{strong1}\\
\nabla\phi^\eps &\to \nabla\phi^*
&\quad& \text{strongly\,\, in\,\, } C(\mathbb{R}_+;L^{2})\cap L^2(\mathbb{R}_+;L^2), \label{strong10}\\
q^\eps &\to q^*
&\quad& \text{strongly\,\, in\,\, } L^2(1,\infty;H^{m-1}), \label{strong2}\\
q^\eps+\nabla p(\rho^\eps)+\rho^\eps\nabla\phi^\eps
&\to 0
&\quad& \text{strongly\,\, in\,\, } L^2(1,\infty;H^{m-1}). \label{strong3}
\end{alignat}
\end{theorem}


\begin{remark}
Some comments on Theorem \ref{theorem1.3} are in order.
\begin{itemize}
	
\item If \eqref{1.11} holds, condition \eqref{1.12} is always satisfied in the relaxation limit $\eps\to 0$. Therefore, Theorem \ref{theorem1.3} is valid for arbitrary large initial data.

\item The relaxation limit holds for ill-prepared large initial data. Indeed, Darcy's law is not required to hold at the initial time $t=0$. More precisely, \eqref{F1.5} implies that
\[   q^*(0,\cdot)=-\nabla p(\rho_0)-\rho_0\nabla\phi_0,   \]
which is different from  the limit of $q^\eps|_{t=0}=\frac{1}{\eps}q_0$. Therefore, there is
an initial layer for variable $q$. The initial layer corrections $q_L^\eps$ and $Z_L^\eps$ play a crucial role in addressing ill-prepared initial data.


\item It is worth emphasizing that the error estimate \eqref{error2} provides 
 an exponential time-decay rate. This estimate not only captures the stability of the relaxation limit, but also reflects the large-time behavior of solutions. 
To the best of our knowledge, Theorem \ref{theorem1.3} appears to be a new attempt for the Euler--Poisson system.

\item Last but not least, we consider the case where $\rho_0=\rho_0^*$ and $q_0$ are independent of $\eps$. In fact, our results hold for the more general initial data, provided that $\rho_0^\eps$, $q_0^\eps$ and $\rho_0^*$ satisfy \eqref{1.11}-\eqref{1.12}, \eqref{DD3.1} and additionally $\|\rho_0^\eps-\rho_0^*\|_{H^{m-1}}=\mathcal{O}(\eps)$.
\end{itemize}
\end{remark}


\vspace{2mm}

 We briefly explain the main ideas of the proof. Compared with the isothermal Euler system without the Poisson term, the major difficulty comes from the nonlocal nonlinear drift
\[
\dive\bigl(\rho \nabla(-\Delta)^{-1}(\rho-1)\bigr),
\]
which occurs both in the Euler--Poisson system \eqref{Euler} and in its limiting drift--diffusion system \eqref{F1.6}.



Our proof strongly relies on a new large-data global well-posedness result for  the limiting drift--diffusion system \eqref{F1.6} (Theorem \ref{theorem1.2}). The key observation is that the nonlinear and nonlocal drift term is compatible with the maximum principle (Lemma \ref{L2.1}). Once the density is bounded from above and below, we show that the nonlinear drift--diffusion operator is coercive, which leads to the $H^1$-energy estimate (Lemmas \ref{L2.2}-\ref{L2.3}). To derive higher-order a priori estimates, we use commutator estimates to put the nonlinear terms into a Gronwall-type form. The coefficient appearing in this inequality can be controlled when $d\leq5$, by the $L^2(0,T;H^2)$ dissipation estimate and the Gagliardo–Nirenberg inequality, which allows us to close the higher-order estimates by an induction argument (Lemma \ref{L2.4}).

For the Euler--Poisson system, 
the entropy inequality gives the basic \(L^2\) control and the relaxation term provides the strong dissipation for momentum (\ref{lem:L2}). However, the Poisson force creates an extra difficulty in the $H^{m-1}$-energy estimates: there is a Gr\"onwall factor of the form
\[
\exp\left(C\int_0^t \|\nabla\phi(t')\|_{L^\infty}^2\,dt'\right),
\]
see Lemma \ref{lem:Hk-refined}. In order to eliminate it,
we recover the drift--diffusion structure for the density. Combining the mass equation with the momentum equation, we have
\[
\partial_t \rho-\eps a^2 \Delta \rho-\eps \dive\big(\rho \nabla \Lambda^{-2}(\rho-1)\big)
=\eps\partial_t \dive q+\eps \dive\dive\big(\frac{q\otimes q}{\rho}\big).
\]
The left-hand side is precisely the drift--diffusion operator appearing in the limiting system (after the transformation $t\rightarrow \frac{t}{\eps}$), while the right-hand side contains only relaxation and convection remainders. Hence, the energy method for the drift--diffusion can be employed, up to those remainder terms, to recover the intrinsic
 density dissipation  (Lemma \ref{lemma4.4}), which  gives the required control of $\|\nabla\phi\|_{L^2(0,T;L^\infty)}$ and thus closes the $H^{m-1}$-energy estimates.  The $H^{m}$-energy estimate is  performed by using  
 hyperbolic symmetrization and avoiding loss of derivatives (Lemma \ref{lemma45}).


It remains to prove that the density of the Euler--Poisson system is uniformly bounded from below and above. Since no maximum principle is available at the hyperbolic level, our strategy is to transfer the bounds from the limiting drift--diffusion system to the Euler--Poisson system via the error estimates on  $\rho^\eps-\rho^*$. That step strongly relies on the maximum principle for the limiting solution $\rho^*$. More precisely, we compare $\rho^\eps$ with $\rho^*$ by introducing the variable from Darcy's law
\[
Z^\eps
:=q^\eps+a^2\nabla\rho^\eps+\rho^\eps\nabla\Lambda^{-2}(\rho^\eps-1)-Z_L^\eps,
\]
where the initial layer correction $Z_L^\eps$ is defined in \eqref{qL}. Then $\rho^\eps$ satisfies
\[
\partial_t \rho^\eps-a^2 \Delta \rho^\eps
-\dive\big(\rho^\eps \nabla \Lambda^{-2}(\rho^\eps-1)\big)
=-\dive Z^\eps-\dive Z_L^\eps.
\]
Subtracting the limiting equation for $\rho^*$, we obtain a drift--diffusion type equation for the density error $\rho^\eps-\rho^*$, with source terms generated by $Z^\eps$ and $Z_L^\eps$. The variable $Z^\eps$ has zero initial value and is strongly damped, hence it is of order $\mathcal O(\eps)$ in $L^2(0,T;H^{m-1})$. The main difficulty comes from the initial layer
\[
Z_L^\eps=q_L^\eps+e^{-\frac{t}{\eps^2}} (\nabla p(\rho_0)+\rho_0 \nabla\phi_0)\quad\text{with the singular part}\,\, q_L^\eps=e^{-\frac{t}{\eps^2}}\frac{1}{\eps}q_0.
\]
Although $q_L^\eps$ is only $\mathcal O(1)$ in $L^2(\mathbb{R}_+;H^m)$, it is $\mathcal O(\eps)$ in $L^1(\mathbb{R}_+;H^m)$. We can prove that the $L^1$-time integrability is sufficient to close the density error estimate, so that no additional diffusive initial layer is needed. Since $\rho^*$ enjoys positive lower and upper bounds by the maximum principle, the smallness of $\rho^\eps-\rho^*$ in $H^{m-1}$ 
yields the same bounds for $\rho^\eps$ when $\eps E_0$ is sufficiently small. See Lemmas \ref{lemma4555}--\ref{lem2.6} for more details. Finally, combining a priori estimates closes the bootstrap under the double-exponential smallness condition
\[
e^{Ce^{CE_0^2}}E_0\eps\ll1.
\]


\vspace{2mm}


Also, we would like to mention another line of research concerning the global smooth irrotational solutions for the Euler--Poisson equations in the absence of damping. See Guo \cite{Guo1}, Guo-Pausader \cite{Guo2} and Germain-Masmoudi-Pausader \cite{GMP1}
in three spatial dimensions, and Ionescu-Pausader \cite{IP1} and Li-Wu \cite{LW1} in two spatial dimensions. These solutions exist globally in time, remain regular for small initial perturbations, and exhibit dispersive decay thanks to the self-consistent electrostatic interactions. The rigorous construction and analysis of such flows, particularly in three spatial dimensions, form a key foundation for the mathematical study of plasma stability, kinetic-to-fluid model justification, and nonlinear dispersive PDE methods.


\vspace{2mm}

 The rest of the paper is organized as follows. Section \ref{section3} collects some definitions and technical lemmas that are frequently used in the proof. In Section \ref{sectionDD}, we develop the maximum principle for the drift--diffusion system and give the proof of Theorem \ref{theorem1.2}. 
 Section \ref{section:EP1} is devoted to uniform energy estimates for the solutions to the Euler--Poisson system. Section \ref{section:EP2} involves error estimates and lower and upper bounds of the density. 
 We close the uniform {\emph{a priori}} estimates and obtain the global existence part of Theorem \ref{theorem1.1} in Section
 \ref{section:EP3}. Finally, the enhanced stability of the error estimate and the proof of Theorem \ref{theorem1.3} will be presented in Section \ref{section:error}.

\vspace{5mm}
\section{Preliminaries}\label{section3}

Throughout this paper, $C>0$ denotes a harmless constant independent of $t$ and $\varepsilon$. The notation $\mathcal{F}(f)$ and  $\mathcal{F}^{-1}(f)$ stand for the Fourier transform and the inverse Fourier transform of the function $f$, and we write
$$
\Lambda^{\sigma}:= (-\Delta)^{\frac{\sigma}{2}}=\mathcal{F}^{-1} \Big( |\xi|^{\sigma} \mathcal{F}(\cdot ) \Big),\quad\sigma\in\mathbb{R}.
$$
In the case $\sigma=2$, we have $-\Delta=\Lambda^{2}$.
For $s\in\mathbb{R}$, we denote by $H^s$ the Sobolev space
of exponent $s$ with the standard norm
$$
\|f\|_{H^s}:=\Big(\int_{\mathbb{R}^d} (1+|\xi|^2)^s|\mathcal{F}(f) (\xi)|^2 \,d\xi \Big)^{\frac{1}{2}}<\infty.
$$
In particular, when $s$ is a positive integer, we have the equivalence
$$
\|f\|_{H^s}^2\sim \sum_{0\leq |\alpha|\leq s}\|\partial^{\alpha} f\|_{L^2}^2.
$$
Furthermore, we denote by $\dot{H}^s(\mathbb{R}^d)$
the homogeneous Sobolev space endowed with the norm
$$
\|f\|_{\dot{H}^s}:=\Big(\int_{\mathbb{R}^d} |\xi|^{2s}|\mathcal{F}(f) (\xi)|^2 \,d\xi \Big)^{\frac{1}{2}}= \|\Lambda^{s}f\|_{L^2}.
$$

The following Sobolev and Moser-type inequalities can be found in
\cite{11}. They will be repeatedly used in the proof.
\begin{lemma}
The following estimates hold.
\vspace{1mm}
\begin{itemize}
\item Let $s>\frac{d}{2}$. The embedding from $H^{s}$ to
$L^\infty\cap C^0$ is continuous and for $u\in H^{s}$,
\begin{equation}
\label{F1.12}
   \|u\|_{L^\infty}\le C\|u\|_{H^{s}}.
\end{equation}

\item Let $-\frac{d}{2}<s<0$ and $q=\frac{2d}{d-2s}$. The embedding from $L^q$ to $H^{s}$ is continuous and for $u\in L^q$,
\begin{equation}
\label{F1.120}
   \|u\|_{H^s}\le C\|u\|_{L^q}.
\end{equation}

\item Let $s_1,s_2\in \mathbb{R}$. The operator $\Lambda^{s_1}$ is an isomorphism from $H^{s_2}$ to $H^{s_2-s_1}$.

\item For
$u, v\in \dot{H}^{s}\cap L^\infty$ with $s>0$,
\begin{equation}
   \|uv\|_{\dot{H}^{s}}\le C \|u\|_{L^{\infty}}\|v\|_{\dot{H}^s}+C\|u\|_{\dot{H}^s}\|v\|_{L^{\infty}}.\label{F1.1311}
\end{equation}
Consequently,  the Sobolev space $H^{s}$ with $s>\frac{d}{2}$ is an algebra, and we have
\begin{equation}
\label{F1.13}
  \|uv\|_{H^{s}}\le C \|u\|_{H^{s}}\|v\|_{H^s}.
\end{equation}


\end{itemize}

\end{lemma}

We also state commutator estimates that are widely used in analyzing the highest-order derivative (cf. \cite{KatoPonce1988}). 
\begin{lemma}
For any integer $k\geq 1$, if $\nabla u\in L^\infty\cap H^{k-1}$ and
$v\in L^\infty\cap H^{k-1}$, then it holds
\begin{equation}
\label{F1.14}
  \sum_{|\alpha|=k}
  \|\partial^{\alpha}(uv)-u\partial^{\alpha}v\|_{L^2}
  \leq
  C
  \|\nabla u\|_{L^\infty}\|v\|_{H^{k-1}}
  +C\|\nabla u\|_{H^{k-1}}\|v\|_{L^\infty}.
\end{equation}
\end{lemma}

We state the lemma concerning nonlinear composition estimates in Sobolev spaces (see \cite{Taylor1996}).
\begin{lemma}
Let $g$ be a smooth function on a compact set
$D\subset\mathbb{R}^n$ and let $k\ge 1$ be an integer. For $u\in H^k\cap L^\infty$
satisfying $u(x)\in D$,
\begin{equation}
\label{F1.15}
   \sum_{|\alpha|=k} \|\partial^{\alpha} g(u)\|_{L^2}\le C_{k,D,g}\bigl(1+\|u\|_{L^\infty}^{k-1}\bigr)\|u\|_{\dot{H}^k}.
\end{equation}

\end{lemma}

Finally, we recall the Gagliardo-Nirenberg inequality.
\begin{lemma}
\label{Lt9}
    Assume that $q,r$ satisfy $1\leq q,r\le \infty$ and $j,m\in \mathbb{Z}$
satisfy $0\le j<m$. For any $f\in C_0^{\infty}(\mathbb{R}^d)$, we have
	\begin{equation}
	\|D^jf\|_{L^p}\leq
C\|D^mf\|_{L^r}^{\alpha}\|f\|_{L^q}^{1-\alpha}\label{w11},
	\end{equation}
where
\[  \frac{1}{p}-\frac{j}{d}
=\alpha\Big(\frac{1}{r}-\frac{m}{d}\Big)+(1-\alpha)\frac{1}{q},\quad
\frac{j}{m}\le\alpha\le1,  \]
and $C$ depends on $m,d,j,q,r,\alpha$.
\end{lemma}


\section{Global well-posedness of the drift--diffusion system}\label{sectionDD}
\vspace{2mm}

In this section, we consider the Cauchy problem for the one-fluid isothermal
drift--diffusion system \eqref{F1.6}.
Since the electric field is given by \eqref{nablaPhi}, it suffices to solve the reformulated drift--diffusion system \eqref{DDp}.

We prove the large-data global existence and uniqueness of classical solutions to the Cauchy problem \eqref{DDp}. 
\begin{theorem}\label{theorem1.2}
Let $1\le d\le 5$ and $m>\frac{d}{2}+1$ be an integer. Assume that $\rho_0^*$ satisfies
\begin{equation}
\begin{aligned}\label{DD3.1}
\rho_0^*-1\in H^{m-1},\quad\quad \rho_1\leq \rho_0^*(x)\leq \rho_2,\quad x\in\mathbb{R}^d,
\end{aligned}
\end{equation}
with two positive constants $\rho_1\leq \rho_2$. Then, system \eqref{DDp} admits a unique global classical solution $\rho^*$ that satisfies 
\begin{equation}\label{1.16}
\begin{aligned}
&\rho^*-1\in C(\mathbb{R}_+;H^{m-1})\cap L^2(\mathbb{R}_+;H^{m}), \quad\quad \rho_1\leq \rho^*(t,x)\leq \rho_2,\quad (t,x)\in \mathbb{R}_+\times\mathbb{R}^d,
\end{aligned}
\end{equation}
and
\begin{equation}\label{exp1}
\begin{aligned}
\sup_{t\in\mathbb{R}_+}e^{\lambda_0^* t}\|\rho^*(t)-1\|_{H^{m-1}}^2+\int_0^\infty e^{\lambda_0^* t}\|\rho^*(t)-1\|_{H^{m}}^2\,dt \leq C_0^* \|\rho_0^*-1\|_{H^{m-1}}^2.
\end{aligned}
\end{equation}
Additionally, if $\phi^*$ is given by \eqref{nablaPhi} and  $\grad\phi^*_0:=\grad(-\Delta)^{-1}(\rho_0^*-1)\in L^2$, then it holds
\begin{align}\label{1.19}
\sup_{t\in\mathbb{R}_+}e^{\lambda_0^* t}\|\nabla\phi^*(t)\|_{L^2}^2+\int_0^\infty e^{\lambda_0^* t}\|\nabla\phi^*(t)\|_{L^2}^2\,dt\leq  C_0^*\|\nabla\phi^*_0\|_{L^2}^2.
\end{align}
Here, $\lambda_0^*$ and $C_0^*>0$ are two constants.
\end{theorem}

For simplicity, we omit the superscript $^*$ and denote the solution and initial datum to \eqref{DDp} by $\rho$ and $\rho_0$, respectively.

According to the classical theorem for parabolic equations (see \cite{Lady}), there exists a time $T>0$ such that the problem \eqref{DDp} 
has a solution $\rho$ satisfying $\rho-1\in C\bigl([0,T];H^{m-1}(\mathbb{R}^d)\bigr)$ and that $\rho$ is away from zero.
To extend the local solution globally in time, we need to establish {\emph{a priori}} estimates independent of $T$ (see Lemmas \ref{L2.1}-\ref{L2.4} below).

\begin{lemma}\label{L2.1}
(Maximum principle) For all $(t,x)\in [0,T]\times\mathbb{R}^d$, it holds that
\begin{equation}\label{F2.4}
\rho_1\le\rho(t,x)\le\rho_2.
\end{equation}
\end{lemma}

\noindent{\bf Proof.}
Set
\[
\tilde{\rho}:=(\rho-\rho_2)_+=\max(\rho-\rho_2,0)\ge 0.
\]
Testing the first equation of \eqref{F1.6} by $\tilde{\rho}$ yields
\begin{align*}
\frac{d}{dt}\|\tilde{\rho}\|_{L^2}^2+2a^2\|\nabla\tilde{\rho}\|_{L^2}^2
&=2\langle\dive(\rho\nabla\phi),\tilde{\rho}\rangle\\
&=-2\langle\rho\nabla\phi,\nabla\tilde{\rho}\rangle\\
&=-2\langle(\rho-\rho_2)\nabla\phi,\nabla\tilde{\rho}\rangle
-2\rho_2\langle\nabla\phi,\nabla\tilde{\rho}\rangle\\
&=-\langle\nabla\phi,\nabla(\tilde{\rho}^2)+2\rho_2\nabla\tilde{\rho}\rangle\\
&=\langle\Delta\phi,\tilde{\rho}^2+2\rho_2\tilde{\rho}\rangle.
\end{align*}
Using the second equation of \eqref{F1.6}, we obtain
\[
\frac{d}{dt}\|\tilde{\rho}\|_{L^2}^2+2a^2\|\nabla\tilde{\rho}\|_{L^2}^2
=-\langle\rho-1,\tilde{\rho}^2+2\rho_2\tilde{\rho}\rangle.
\]
Moreover,
\begin{align*}
-(\rho-1)\big(\tilde{\rho}^2+2\rho_2\tilde{\rho}\big)
&=-(\rho-\rho_2)\big(\tilde{\rho}^2+2\rho_2\tilde{\rho}\big)
-(\rho_2-1)\big(\tilde{\rho}^2+2\rho_2\tilde{\rho}\big)\\
&=-\tilde{\rho}\big(\tilde{\rho}^2+2\rho_2\tilde{\rho}\big)
-(\rho_2-1)\big(\tilde{\rho}^2+2\rho_2\tilde{\rho}\big)\\
&\le0.
\end{align*}
Hence
\[
\frac{d}{dt}\|\tilde{\rho}\|_{L^2}^2+2a^2\|\nabla\tilde{\rho}\|_{L^2}^2\le 0.
\]
Integrating over $[0,t]$ and using $\tilde{\rho}|_{t=0}=0$ gives
$\|\tilde{\rho}(t)\|_{L^2}=0$, so $\rho(t,x)\le \rho_2$.

Similarly, using $(\rho-\rho_1)_-=\min(\rho-\rho_1,0)\le 0$ as a test function,
we obtain $\rho(t,x)\ge \rho_1$.
\hfill$\Box$


\begin{lemma}\label{L2.2}
($L^2$ estimate) For all $t\in [0,T]$, it holds
\begin{equation}\label{F2.6}
\|\rho(t)-1\|_{L^2}^2+\int_0^t\|\rho(t')-1\|_{H^1}^2\,dt'\le  C\|\rho_0-1\|_{L^2}^2.
\end{equation}
\end{lemma}

\noindent{\bf Proof.}
Multiplying the first equation of \eqref{F1.6} by $2(\rho-1)$ and integrating
over $\mathbb{R}^d$, we obtain
\begin{align*}
\frac{d}{dt}\|\rho-1\|_{L^2}^2+2a^2\|\nabla\rho\|_{L^2}^2
&= 2\int_{\mathbb{R}^d} (\rho-1)\dive (\rho \nabla\phi)\,dx\\
&= 2\int_{\mathbb{R}^d} (\rho-1)\nabla \rho \cdot\nabla\phi\,dx
-2\int_{\mathbb{R}^d} \rho(\rho-1)^2\,dx\\
&= \int_{\mathbb{R}^d} \nabla (\rho-1)^2\cdot\nabla\phi\,dx
-2\int_{\mathbb{R}^d} \rho(\rho-1)^2\,dx\\
&= \int_{\mathbb{R}^d}(\rho-1)^3\,dx
-2\int_{\mathbb{R}^d} \rho(\rho-1)^2\,dx\\
&=-\|\rho-1\|_{L^2}^2-\int_{\mathbb{R}^d}\rho(\rho-1)^2\,dx.
\end{align*}
This yields
\begin{equation}\label{F2.5}
\frac{d}{dt}\|\rho-1\|_{L^2}^2+\|\rho-1\|_{L^2}^2+\int_{\mathbb{R}^d}\rho(\rho-1)^2\,dx
+2a^2\|\nabla \rho\|_{L^2}^2=0.
\end{equation}
Consequently, \eqref{F2.6} holds.
\hfill$\Box$


\begin{lemma}\label{L2.3}
($\dot{H}^1$ estimate) For all $t\in [0,T]$, it holds
\begin{equation}\label{F2.8}
\|\nabla\rho(t)\|^2_{L^2}+\int_0^t\|\nabla \rho(t')\|^2_{H^1}\,dt'\le C\|\rho_0-1\|^2_{H^1}.
\end{equation}
\end{lemma}

\noindent{\bf Proof.}
Note that \eqref{F1.6} yields
\[
\partial_t\partial_{x_j}\rho-a^2\Delta\partial_{x_j}\rho+\partial_{x_j}\rho
=\dive(\partial_{x_j}\rho\nabla\phi+(\rho-1)\nabla\partial_{x_j}\phi),
\]
and
\begin{align*}
\frac{d}{dt}\|\partial_{x_j}\rho\|_{L^2}^2+2a^2\|\nabla\partial_{x_j}\rho\|_{L^2}^2+2\|\partial_{x_j}\rho\|_{L^2}^2
&=-2\langle\partial_{x_j}\rho\nabla\phi+(\rho-1)\nabla\partial_{x_j}\phi,\nabla\partial_{x_j}\rho\rangle\\
&=-\langle\nabla\phi,\nabla(\partial_{x_j}\rho)^2\rangle
-2\langle (\rho-1) \nabla\partial_{x_j}\phi,\nabla\partial_{x_j}\rho\rangle\\
&=-\langle\rho-1,(\partial_{x_j}\rho)^2\rangle
-2\langle (\rho-1) \nabla\partial_{x_j}\phi,\nabla\partial_{x_j}\rho\rangle.
\end{align*}
Clearly,
\[
|\langle\rho-1,(\partial_{x_j}\rho)^2\rangle|
\le (1+\rho_2)\|\partial_{x_j}\rho\|_{L^2}^2
\le C\|\rho-1\|^2_{H^1}.
\]
By the Poisson equation,
\[
\|\nabla\partial_{x_j}\phi\|_{L^2}\le C\|\rho-1\|_{L^2},
\]
hence
\[
2|\langle (\rho-1)\nabla\partial_{x_j}\phi,\nabla\partial_{x_j}\rho\rangle|
\le C\|\rho-1\|_{L^2}\|\nabla\partial_{x_j}\rho\|_{L^2}
\le C\|\rho-1\|_{L^2}^2+a^2\|\nabla\partial_{x_j}\rho\|_{L^2}^2.
\]
Therefore,
\[
\frac{d}{dt}\|\partial_{x_j}\rho\|_{L^2}^2+a^2\|\nabla\partial_{x_j}\rho\|_{L^2}^2+2\|\partial_{x_j}\rho\|_{L^2}^2\le C\|\rho-1\|^2_{H^1},
\]
and summing in $j$ gives
\begin{equation}\label{F2.7}
\frac{d}{dt}\|\nabla\rho\|_{L^2}^2+\min\{2,a^2\}\|\nabla\rho\|_{H^1}^2\le C\|\rho-1\|^2_{H^1}.
\end{equation}
Integrating \eqref{F2.7} over time and using \eqref{F2.6}, we arrive at \eqref{F2.8}.
\hfill$\Box$


\begin{lemma}\label{L2.4}
(Higher-order estimates)
Let $1\le d\le 5$ and $2\le k\le m-1$. Then, for all $t\in[0,T]$, it holds that
\begin{equation}\label{F2.9}
\|\rho(t)-1\|_{\dot H^k}^2
+\int_0^t \|\nabla \rho(t')\|_{\dot H^k}^2\,dt'
\le
e^{C\|\rho_0-1\|_{H^2}^2}
\|\rho_0-1\|_{\dot H^k}^2.
\end{equation}
\end{lemma}

\begin{proof}
Applying $\partial^\alpha$ with $|\alpha|=k$ to \eqref{F1.6}, we obtain
\begin{align*}
\partial_t \partial^\alpha \rho
-a^2 \Delta \partial^\alpha \rho
+\partial^\alpha \rho
&=\partial^\alpha \operatorname{div}\big((\rho-1)\nabla \phi\big) \\
&=\sum_{j=1}^d (\partial_{x_j}\partial^\alpha \rho)\,\partial_{x_j}\phi
+\sum_{j=1}^d \Bigl[
\partial_{x_j}\partial^\alpha\big((\rho-1)\,\partial_{x_j}\phi\big)
-(\partial_{x_j}\partial^\alpha \rho)\,\partial_{x_j}\phi
\Bigr].
\end{align*}
Taking the $L^2$ inner product with $\partial^\alpha\rho$, we deduce
\begin{align*}
\frac{d}{dt}\|\partial^\alpha \rho\|_{L^2}^2
+2a^2\|\nabla \partial^\alpha \rho\|_{L^2}^2
+2\|\partial^\alpha \rho\|_{L^2}^2
= I_{1,\alpha}+I_{2,\alpha},
\end{align*}
where
\begin{align*}
I_{1,\alpha}
&:=2\sum_{j=1}^d
\langle
(\partial_{x_j}\partial^\alpha \rho)\,\partial_{x_j}\phi,
\partial^\alpha \rho
\rangle,\\
I_{2,\alpha}&
:=2\sum_{j=1}^d
\langle
\partial_{x_j}\partial^\alpha\big((\rho-1)\,\partial_{x_j}\phi\big)
-(\partial_{x_j}\partial^\alpha \rho)\,\partial_{x_j}\phi,
\partial^\alpha \rho
\rangle.
\end{align*}

We first estimate $I_{1,\alpha}$. By integration by parts and the Poisson equation $-\Delta\phi=\rho-1$, we have
\begin{align*}
I_{1,\alpha}
&=\sum_{j=1}^d
\langle
\partial_{x_j}\phi,
\partial_{x_j}\big((\partial^\alpha \rho)^2\big)
\rangle
=-\langle
\Delta\phi,
(\partial^\alpha \rho)^2
\rangle 
\le
\|\rho-1\|_{L^\infty}
\|\partial^\alpha \rho\|_{L^2}^2.
\end{align*}
Choose $s>\frac d2$. Since $H^s\hookrightarrow L^\infty$, it follows that
\begin{equation}\label{I1-bound}
|I_{1,\alpha}|
\le
C\|\rho-1\|_{H^s}
\|\partial^\alpha \rho\|_{L^2}^2.
\end{equation}

Next, we estimate $I_{2,\alpha}$. By the commutator estimate \eqref{F1.14},
for $|\alpha|=k$,
\begin{align*}
&\sum_{j=1}^d
\left\|
\partial_{x_j}\partial^\alpha\big((\rho-1)\,\partial_{x_j}\phi\big)
-(\partial_{x_j}\partial^\alpha \rho)\,\partial_{x_j}\phi
\right\|_{L^2} \\
&\quad\le
C\Big(
\|\nabla^2\phi\|_{L^\infty}\|\rho-1\|_{\dot H^k}
+\|\nabla^2\phi\|_{\dot H^k}\|\rho-1\|_{L^\infty}
\Big)\\
&\quad\le
C\|\rho-1\|_{H^3}\|\rho-1\|_{\dot H^k},
\end{align*}
where we used $d\le5$, $H^3\hookrightarrow L^\infty$, and
$\|\nabla^2\phi\|_{\dot H^k}\le C\|\rho-1\|_{\dot H^k}$. Therefore,
\begin{equation}\label{I2-bound}
|I_{2,\alpha}|
\le
C
\|\rho-1\|_{H^3}
\|\rho-1\|_{\dot H^k}
\|\partial^\alpha \rho\|_{L^2}.
\end{equation}

Summing \eqref{I1-bound} and \eqref{I2-bound} over all multi-indices $|\alpha|=k$, we infer that
\begin{align}\label{maomao}
&\frac{d}{dt}\|\rho-1\|_{\dot H^k}^2
+2a^2\|\nabla\rho\|_{\dot H^k}^2
+2\|\rho-1\|_{\dot H^k}^2 \notag\\
&\qquad\le
C\|\rho-1\|_{H^3}\|\rho-1\|_{\dot H^k}^2.
\end{align}

We first consider the case $k=2$. Integrating \eqref{maomao} over $[0,t]$,
we obtain
\begin{align*}
&\|\rho(t)-1\|_{\dot H^2}^2
+2a^2\int_0^t \|\nabla\rho(t')\|_{\dot H^2}^2\,dt'
+2\int_0^t \|\rho(t')-1\|_{\dot H^2}^2\,dt' \\
&\qquad\le
\|\rho_0-1\|_{\dot H^2}^2
+C\int_0^t
\|\rho(t')-1\|_{H^3}\|\rho(t')-1\|_{\dot H^2}^2\,dt'.
\end{align*}
Since
\[
\|\rho-1\|_{H^3}
\le C\|\rho-1\|_{H^2}+C\|\rho-1\|_{\dot H^3},
\]
Young's inequality gives, for any $\eta>0$,
\begin{align*}
C\|\rho-1\|_{H^3}\|\rho-1\|_{\dot H^2}^2
&\le
C\|\rho-1\|_{H^2}\|\rho-1\|_{\dot H^2}^2
+\eta\|\rho-1\|_{\dot H^3}^2
+C_\eta\|\rho-1\|_{\dot H^2}^4.
\end{align*}
Taking $\eta>0$ sufficiently small and using the previously established
lower-order estimates \eqref{F2.6} and \eqref{F2.8}, together with Gr\"onwall's inequality,
we deduce
\begin{align}\label{3.131}
\|\rho(t)-1\|_{\dot H^2}^2
+\int_0^t \|\nabla\rho(t')\|_{\dot H^2}^2\,dt'
\le
Ce^{C\|\rho_0-1\|_{H^2}^2}
\|\rho_0-1\|_{\dot H^2}^2.
\end{align}

We now turn to the case $3\le k\le m-1$. Returning to \eqref{maomao}, we obtain
\begin{align*}
&\|\rho(t)-1\|_{\dot H^k}^2
+2a^2\int_0^t \|\nabla\rho(t')\|_{\dot H^k}^2\,dt'
+2\int_0^t \|\rho(t')-1\|_{\dot H^k}^2\,dt' \\
&\qquad\le
\|\rho_0-1\|_{\dot H^k}^2
+C\int_0^t \|\rho(t')-1\|_{H^3}\|\rho(t')-1\|_{\dot H^k}^2\,dt' \\
&\qquad\le
\|\rho_0-1\|_{\dot H^k}^2
+\eta\int_0^t \|\rho(t')-1\|_{\dot H^k}^2\,dt'
+C_\eta\int_0^t \|\rho(t')-1\|_{H^3}^2\|\rho(t')-1\|_{\dot H^k}^2\,dt'.
\end{align*}
Taking $\eta>0$ sufficiently small and using \eqref{F2.6}, \eqref{F2.8} and \eqref{3.131}, we infer from Gr\"onwall's inequality that \eqref{F2.9} holds. This completes the proof.
\end{proof}

\noindent
{\emph{\textbf{Proof of Theorem \ref{theorem1.2}}}}.
Let $\rho$ be a local solution to \eqref{F1.6} on $[0,T]\times\mathbb{R}^d$. Then Lemmas \ref{L2.1}--\ref{L2.4} provide uniform a priori estimates, which prevent any blow-up of the solution norm in the regularity class under consideration. By the standard continuation argument, the local solution can therefore be extended globally in time. This proves the global existence part of Theorem \ref{theorem1.2}.

We next establish the exponential decay estimate \eqref{exp1}. Taking a sufficiently large multiple of \eqref{F2.5} and adding it to
\eqref{F2.7}, we find a positive constant $c_1$ such that
\begin{align}\label{H1exp}
\frac{d}{dt}\|\rho-1\|_{H^1}^2+2c_1\|\rho-1\|_{H^2}^2\le 0.
\end{align}
Multiplying \eqref{H1exp} by $e^{c_1 t}$ and integrating over $[0,t]$, we obtain
\begin{align}\label{exppp}
e^{c_1 t}\|\rho(t)-1\|_{H^1}^2
+c_1\int_0^t e^{c_1 t'}\|\rho(t')-1\|_{H^2}^2\,dt'
\le \|\rho_0-1\|_{H^1}^2.
\end{align}
If $m-1\ge 2$, then taking $k=m-1$ in \eqref{maomao} and using Young's inequality on the right-hand side of \eqref{maomao}, we have
\begin{align*}
\frac{d}{dt}\|\rho-1\|_{\dot H^k}^2
+2a^2\|\nabla\rho\|_{\dot H^k}^2
+\|\rho-1\|_{\dot H^k}^2\le
C\|\rho-1\|_{H^3}^2
\|\rho-1\|_{\dot H^k}^2.
\end{align*}
For some small constant $c_2>0$, multiplying the above inequality by
\[
e^{c_2t-C\int_0^t \|\rho(t')-1\|_{H^{3}}^2\,dt'},
\]
we infer that
\begin{align*}
&\frac{d}{dt}\bigg(
e^{c_2t-C\int_0^t \|\rho(t')-1\|_{H^{3}}^2\,dt'}
\|\rho-1\|_{\dot H^k}^2
\bigg)
+c_2 e^{c_2 t-C\int_0^t\|\rho(t')-1\|_{H^{3}}^2\,dt'}
\|\rho-1\|_{\dot H^k}^2
\le 0.
\end{align*}
Integrating the above inequality over $[0,t]$, and then combining the resulting estimate with \eqref{exppp}, we obtain \eqref{exp1}.

Finally, we derive the regularity and decay estimate for the electric field under the extra assumption $\nabla\phi_0\in L^2$. Taking the $L^2$ inner product of \eqref{DDp} with $\Lambda^{-2}(\rho-1)$ yields
\begin{equation}\label{D1}
\begin{aligned}
&\frac{d}{dt}\|\Lambda^{-1}(\rho-1)\|_{L^2}^2
+2a^2\|\Lambda^{-1}\nabla(\rho-1)\|_{L^2}^2 \\
&\qquad
=-2\int_{\mathbb{R}^d}\rho\,\nabla\phi\cdot\Lambda^{-2}\nabla(\rho-1)\,dx
=-2\int_{\mathbb{R}^d}\rho\,|\nabla\phi|^2\,dx
\le -2\rho_1\|\nabla\phi\|_{L^2}^2,
\end{aligned}
\end{equation}
where we have used the lower bound in \eqref{F2.4}. Since
\[
\|\Lambda^{-1}(\rho-1)\|_{L^2}\sim \|\nabla\phi\|_{L^2},
\qquad
\|\Lambda^{-1}\nabla(\rho-1)\|_{L^2}\sim \|\rho-1\|_{L^2},
\]
it follows from \eqref{D1} that \eqref{1.19} holds. This completes the proof of Theorem \ref{theorem1.2}.

\section{A priori estimates for the Euler--Poisson system}\label{section:EP1}

We establish uniform a priori estimates for solutions to the Cauchy problem \eqref{Euler}--\eqref{Eulerd} on $[0,T_\eps)$, where $T_\eps=T_{\eps}(\eps,\rho_0, q_0)$ is the maximal time of local well-posedness. 

Let $\tilde\rho_2\ge \tilde\rho_1>0$ be two constants such that
\begin{align}\label{assumupperlow}
\tilde\rho_1\le \rho(t,x)\le \tilde\rho_2,\quad \forall (t,x)\in[0,T_\eps)\times\R^d.
\end{align}
To justify certain Sobolev estimates (in particular those involving highest-order derivatives), a standard regularization procedure is required. 
For simplicity of presentation, we shall always assume that $(\rho,v,\phi)$ is a smooth solution to the Cauchy problem \eqref{Euler}--\eqref{Eulerd}.

The proof is divided into five steps.

\subsection{\texorpdfstring{$L^2$}{L2} energy estimate}

We first establish an $L^2$ bound on the solution. 
\begin{lemma}\label{lem:L2}
Under the assumption \eqref{assumupperlow}, we have for all $t\in [0,T_\eps)$ that
\begin{align}\label{L2}
\sup_{t'\in[0,t]} \|(\rho-1,q,\nabla\phi)(t')\|_{L^2}^2+\frac{1}{\eps}\int_0^t \|q(t')\|_{L^2}^2\,dt'\le C \|(\rho_0-1,q_0,\nabla\phi_0)\|_{L^2}^2.
\end{align}
\end{lemma}

\begin{proof}
Introduce the entropy–entropy flux pair
\begin{align}
E(\rho,v)&=\frac12\rho|v|^2+a^2\big(\rho(\ln\rho-1)+1\big),\quad v=\frac{q}{\rho},\\
F(\rho,v)&=\frac12\rho|v|^2 v+a^2\rho v\ln\rho.
\end{align}
Since $\tilde\rho_1\le \rho\le \tilde\rho_2$, a Taylor expansion around $\rho=1$ yields
\begin{align}
\frac{a^2}{2\tilde\rho_2}(\rho-1)^2
\le a^2\big(\rho(\ln\rho-1)+1\big)
\le \frac{a^2}{2\tilde\rho_1}(\rho-1)^2.
\label{entropy-equivalence}
\end{align}
Consequently, one has the equivalence
\begin{align}
\frac{\tilde\rho_1}{2}|v|^2+\frac{a^2}{2\tilde\rho_2}(\rho-1)^2
\le E(\rho,v)
\le \frac{\tilde\rho_2}{2}|v|^2+\frac{a^2}{2\tilde\rho_1}(\rho-1)^2.
\label{energy-equivalence}
\end{align}
For smooth solutions to \eqref{Euler}, it holds that
\begin{align}
\partial_t E(\rho,v)+\dive F(\rho,v)
+\frac1\varepsilon \rho|v|^2+\rho v\cdot \nabla \phi=0.
\label{entropy-identity}
\end{align}
Next, using the continuity equation and the Poisson equation, we compute
\begin{align}
\int_{\R^d} \rho v\cdot \nabla \phi\,dx
&=-\int_{\R^d} \dive(\rho v)\,\phi\,dx
=\int_{\R^d} \partial_t \rho\,\phi\,dx \nonumber\\
&=\int_{\R^d} (-\Delta \partial_t \phi)\,\phi\,dx
=\int_{\R^d} \nabla \partial_t \phi \cdot \nabla \phi\,dx
=\frac12\frac{d}{dt}\int_{\R^d} |\nabla\phi|^2\,dx.
\end{align}
Thus, integrating \eqref{entropy-identity} over $\R^d$, we obtain
\begin{align}\label{3.13}
\frac{d}{dt}\int_{\R^d}\Big(E(\rho,v)+\frac12|\nabla\phi|^2\Big)\,dx
+\frac1\varepsilon \int_{\R^d}\rho|v|^2\,dx=0,
\end{align}
Since $\tilde\rho_1\le \rho\le \tilde\rho_2$, this leads to \eqref{L2}.
\end{proof}

\subsection{\texorpdfstring{$H^{m-1}$}{Hm-1} energy estimate}

Using \eqref{F1.1} and the relation $q=\rho v$, we rewrite system \eqref{Euler} as
\begin{equation}\label{eq:reform}
\left\{
\begin{aligned}
&\partial_t\rho+\dive q=0,\\
&\partial_t q+a^2\nabla\rho+\frac{1}{\varepsilon}q+\nabla\Lambda^{-2}(\rho-1)=f,
\end{aligned}\right.
\end{equation}
where
\begin{equation}\label{def:f-split}
f=f_1+f_2,\qquad 
f_1:=-(\rho-1)\nabla\Lambda^{-2}(\rho-1),\qquad 
f_2:=-\dive(q\otimes v).
\end{equation}

Let $\phi:=\Lambda^{-2}(\rho-1)$. Then one has the identity
\begin{equation}\label{f1-stress}
f_1=(\Delta\phi)\nabla\phi
=\dive\Big(\nabla\phi\otimes\nabla\phi-\frac12|\nabla\phi|^2\mathbf{I}_d\Big).
\end{equation}

\begin{lemma}\label{lem:Hk-refined}
Assume that \eqref{assumupperlow} holds. 
For any $1\le k\le m-1$ and all $t\in[0,T_\eps)$, it holds that
\begin{align}
&a^2\|\rho(t)-1\|_{H^k}^2+\|q(t)\|_{H^k}^2
+\frac{1}{\varepsilon}\int_0^t\|q(t')\|_{H^k}^2\,dt'\nonumber\\
&\le a^2\|\rho_0-1\|_{H^k}^2+\|q_0\|_{H^k}^2
+C\varepsilon\int_0^t \|q(t')\|_{H^m}^2 \|v(t')\|_{H^m}^2 \,dt'\nonumber\\
&\quad+C\varepsilon\int_0^t\|\nabla\phi(t')\|_{L^\infty}^2
\|\rho(t')-1\|_{H^{k}}^2\,dt'.
\label{Hk-energy-int}
\end{align}
\end{lemma}

\begin{proof}
Fix $|\alpha|=k$. Applying $\partial^\alpha$ to \eqref{eq:reform}, we get
\[
\partial_t\partial^\alpha \rho+\dive\partial^\alpha q=0,\qquad
\partial_t\partial^\alpha q+a^2\nabla\partial^\alpha \rho+\frac{1}{\varepsilon}\partial^\alpha q+\nabla\Lambda^{-2}\partial^\alpha \rho=\partial^\alpha f.
\]
Taking the $L^2$ inner product with $(a^2\partial^\alpha \rho,\partial^\alpha q)$ yields
\begin{align*}
\frac12\frac{d}{dt}\Big(a^2\|\partial^\alpha \rho\|_{L^2}^2+\|\partial^\alpha q\|_{L^2}^2\Big)
+\frac{1}{\varepsilon}\|\partial^\alpha q\|_{L^2}^2
&=-a^2\langle\dive\partial^\alpha q,\partial^\alpha \rho\rangle
-a^2\langle\nabla\partial^\alpha \rho,\partial^\alpha q\rangle\\
&\quad-\langle\nabla\Lambda^{-2}\partial^\alpha \rho,\partial^\alpha q\rangle
+\langle\partial^\alpha f,\partial^\alpha q\rangle.
\end{align*}
The first two terms cancel by integration by parts. For the Poisson term, we have
\[
-\langle\nabla\Lambda^{-2}\partial^\alpha \rho,\partial^\alpha q\rangle
=\langle\Lambda^{-2}\partial^\alpha \rho,\dive\partial^\alpha q\rangle
=-\langle\Lambda^{-2}\partial^\alpha \rho,\partial_t\partial^\alpha \rho\rangle
=-\frac12\frac{d}{dt}\|\Lambda^{-1}\partial^\alpha \rho\|_{L^2}^2.
\]
Therefore,
\begin{equation}\label{pp1}
\begin{aligned}
&\quad\frac{d}{dt}\Big(a^2\|\partial^\alpha\rho\|_{L^2}^2+\|\partial^\alpha q\|_{L^2}^2
+\|\Lambda^{-1}\partial^\alpha \rho\|_{L^2}^2\Big)
+\frac{2}{\varepsilon}\|\partial^\alpha q\|_{L^2}^2\\[2mm]
&=2\langle\partial^\alpha f,\partial^\alpha q\rangle\\
&\le \frac{1}{\varepsilon}\|\partial^{\alpha}q\|_{L^2}^2
+C\varepsilon\|\partial^{\alpha}f\|_{L^2}^2.
\end{aligned}
\end{equation}

We now estimate $f_1$ and $f_2$ separately. For $f_2$, the Moser-type product estimate \eqref{F1.1311} yields, for $1\le k\le m-1$,
\begin{align*}
\|f_2\|_{H^k}
&\le C\|q\otimes v\|_{H^{k+1}}
\le C\big(\|q\|_{H^{k+1}}\|v\|_{L^\infty}+\|q\|_{L^\infty}\|v\|_{H^{k+1}}\big).
\end{align*}
Since $m>\frac d2+1$ and $k+1\le m$, we have 
\[
\|q\|_{L^\infty}+\|v\|_{L^\infty}\le C\big(\|q\|_{H^m}+\|v\|_{H^m}\big),
\]
which implies
\begin{equation}\label{est-f2-Hk}
\|f_2\|_{H^k}\le C\|q\|_{H^m}\|v\|_{H^m}.
\end{equation}
For $f_1$, using the structure \eqref{f1-stress} together with the Moser-type estimate \eqref{F1.1311}, we have
\begin{align}
\|f_1\|_{H^k}
&\le C \|\nabla\phi\|_{L^\infty}\|\nabla\phi\|_{H^{k+1}}
\le C\|\nabla\phi\|_{L^\infty}\|\rho-1\|_{H^{k}}.
\label{est-f1-Hk}
\end{align}

Combining \eqref{est-f2-Hk} and \eqref{est-f1-Hk} with \eqref{pp1}, we have the following differential inequality 
\begin{equation}
\begin{aligned}\label{ddtHk}
&\quad\frac{d}{dt}\sum_{|\alpha|=k}\Big(a^2\|\partial^\alpha \rho\|_{L^2}^2
+\|\partial^\alpha q\|_{L^2}^2+\|\Lambda^{-1}\partial^\alpha \rho\|_{L^2}^2\Big)
+\frac{1}{\varepsilon}\sum_{|\alpha|=k}\|\partial^\alpha q\|_{L^2}^2\\
&\le C\varepsilon \|q\|_{H^m}^2 \|v\|_{H^m}^2+C\eps \|\nabla\phi\|_{L^\infty}^2\|\rho-1\|_{H^{k}}^2,
\end{aligned}
\end{equation}
which leads to \eqref{Hk-energy-int}.
\end{proof}

\subsection{Dissipation estimate of \texorpdfstring{$\rho$}{rho} and drift--diffusion structure}

From the momentum equation in \eqref{eq:reform}, we discover
\begin{align}\label{diveq}
   \dive q=-\eps\partial_t\dive q-\eps a^2\Delta\rho
-\eps\dive\big(\rho\nabla\Lambda^{-2}(\rho-1)\big)-\eps\dive\dive(q\otimes v).
\end{align}
Putting \eqref{diveq} into the mass equation of \eqref{eq:reform} allows us to obtain a drift--diffusion structure:
\begin{align}\label{4.3}
\partial_t \rho-\eps a^2 \Delta \rho-\eps \dive\big(\rho \nabla \Lambda^{-2}(\rho-1)\big)= \eps\partial_t \dive q+\eps \dive\dive (q\otimes v).
\end{align}

Consequently, in the spirit of Section \ref{sectionDD}, we can overcome the nonlinear terms caused by the Poisson coupling and establish some dissipation estimates of $\rho$. 
\begin{lemma}\label{lemma4.4}
Assume that \eqref{assumupperlow} holds. For all $t\in [0,T_\eps)$, it holds that
\begin{align}
\varepsilon\int_0^t(\|\rho(t')-1\|_{L^2}^2+\|\nabla\phi(t')\|_{L^2}^2)\,dt'
&\leq C\big(\|\rho_0-1\|_{L^2}^2+\|q_0\|_{L^2}^2+\|\nabla\phi_0\|_{L^2}^2\big)\nonumber\\
&\quad +C\varepsilon\int_0^t \|q(t')\|_{H^{m}}^2\|v(t')\|_{H^1}^2\,dt',\label{rhoL21}
\\
\varepsilon \int_0^t \|\nabla \rho(t') \|_{H^1}^2\,dt'
&\leq C\|\rho_0-1\|_{H^1}^2+C\|\dive q_0\|_{H^1}^2\nonumber\\
&\quad+ C\varepsilon^2 \Big(\sup_{t'\in[0,t]}\|\dive q(t')\|_{H^1}^2
+\frac{1}{\varepsilon} \int_0^t \|\dive q(t')\|_{H^1}^2\,dt'\Big)\nonumber
\\
&\quad+\varepsilon\int_0^t \|q(t')\|_{H^{m}}^2\|v(t')\|_{H^{m}}^2\,dt',\label{rhoL22}
\end{align}
and for $2\leq k\leq m-1$,
\begin{equation}
\begin{aligned}\label{nablarhok}
\varepsilon \int_0^t \|\nabla \rho(t') \|_{H^{k}}^2\,dt'
&\leq Ce^{C\varepsilon \int_0^t \|\nabla\phi(t')\|_{L^\infty}^2\,dt'}
\bigg(\|\rho_0-1\|_{H^{k}}^2+\|\dive q_0\|_{H^{k}}^2\\
&\quad+\varepsilon^2 \Big(\sup_{t'\in[0,t]}\| \dive q(t')\|_{H^{k}}^2
+\frac{1}{\varepsilon} \int_0^t \| \dive q(t')\|_{H^{k}}^2\,dt'\Big)
\\
&\quad+\varepsilon\int_0^t\|q(t')\|_{H^{m}}^2  \|v(t')\|_{H^{m}}^2\,dt'\bigg).
\end{aligned}
\end{equation}
\end{lemma}

\begin{proof}
Taking the $L^2$ inner product of \eqref{4.3} with $\Lambda^{-2}(\rho-1)$, we obtain
\begin{align}
&\frac{1}{2}\|\Lambda^{-1}(\rho-1)\|_{L^2}^2\Big|_{0}^{t}
+\varepsilon\int_{0}^t\Big(a^2\|\Lambda^{-1}\nabla (\rho-1)\|_{L^2}^2
+\int_{\mathbb{R}^d}\rho\,|\Lambda^{-2}\nabla (\rho-1)|^2\,dx\Big)\,dt'
\nonumber\\
&\quad=-\varepsilon \int_0^t \big\langle \partial_t q,\, \Lambda^{-2}\nabla (\rho-1)\big\rangle\,dt'
-\varepsilon \int_0^t \langle \dive (q\otimes v),\, \Lambda^{-2}\nabla (\rho-1) \rangle\,dt'.
\label{energy-phi}
\end{align}

We first handle the time derivative term. Using $\partial_t\rho=-\dive q$ and an integration
by parts in time:
\begin{align}
-\int_0^t \langle \partial_t q, \Lambda^{-2}\nabla (\rho-1)\rangle dt'
&=-\langle q,\Lambda^{-2}\nabla(\rho-1)\rangle\Big|_0^t
+\int_0^t \langle q, \Lambda^{-2}\nabla \partial_t\rho\rangle dt'
\nonumber\\
&=-\langle q,\Lambda^{-2}\nabla(\rho-1)\rangle\Big|_0^t
-\int_0^t \langle q, \Lambda^{-2}\nabla \dive q\rangle dt'.
\end{align}
Noting that
\[
-\langle q, \Lambda^{-2}\nabla \dive q\rangle
=\|\Lambda^{-1}\dive q\|_{L^2}^2,
\]
we get
\begin{align}
-\eps\int_0^t \langle \partial_t q, \Lambda^{-2}\nabla (\rho-1)\rangle dt'
&=-\eps\langle q,\Lambda^{-2}\nabla(\rho-1)\rangle\Big|_0^t
+\eps\int_0^t \|\Lambda^{-1}\dive q(t')\|_{L^2}^2\,dt'.
\end{align}
By $\eps\le1$ and Young's inequality,
\begin{align*}
\eps\big|\langle q,\Lambda^{-2}\nabla(\rho-1)\rangle\big|
\leq \frac{1}{4}\|\Lambda^{-1}(\rho-1)\|_{L^2}^2
+C\|q\|_{L^2}^2.
\end{align*}
Hence, we obtain
\begin{align}
-\eps\int_0^t \langle \partial_t q, \Lambda^{-2}\nabla (\rho-1)\rangle dt'
&\leq \frac{1}{4}\|\Lambda^{-1}(\rho(t)-1)\|_{L^2}^2
+C\big(\|q(t)\|_{L^2}^2+\|q_0\|_{L^2}^2+\|\Lambda^{-2}\nabla (\rho_0-1)\|_{L^2}^2\big)
\nonumber\\
&\quad+C\int_0^t \|q(t')\|_{L^2}^2\,dt'.
\end{align}
For the nonlinear term, Sobolev's embedding $H^{m-1}\hookrightarrow L^\infty$ gives
\begin{align*}
\eps\left|\int_0^t \langle  q\otimes v, \Lambda^{-2}\nabla^2 \rho \rangle dt'\right|
&\leq \eps\int_0^t\|\Lambda^{-2}\nabla^2 \rho(t') \|_{L^2}\|q(t')\|_{L^\infty}\|v(t')\|_{L^2}\,dt'\\
&\leq \frac{\eps}{4}\int_0^t\|\Lambda^{-1}\nabla \rho(t') \|_{L^2}^2\,dt'
+C\eps\int_0^t \|q(t')\|_{H^{m-1}}^2\|v(t')\|_{L^2}^2\,dt'.
\end{align*}
Using the equivalences
\[
\|\Lambda^{-1}(\rho-1)\|_{L^2}\sim \|\nabla\phi\|_{L^2},
\qquad
\|\Lambda^{-1}\nabla\rho\|_{L^2}\sim \|\rho-1\|_{L^2},
\]
and recalling \eqref{L2} and  \eqref{energy-phi}, we deduce  the first estimate \eqref{rhoL21}. 

Next, multiplying \eqref{4.3} by $(\rho-1)$, integrating over $\mathbb{R}^d$ and arguing similarly as in Lemma \ref{L2.2}, we obtain
\begin{equation}\label{rhoL2edd}
\begin{aligned}
&\quad \frac{1}{2}\frac{d}{dt}\|\rho-1\|_{L^2}^2+\frac{a^2}{2}\varepsilon \|\nabla\rho\|_{L^2}^2+\frac{\eps}{2}\|\rho-1\|_{L^2}^2+\frac{\eps}{2}\int_{\mathbb{R}^d}\rho(\rho-1)^2\,dx\\
&\leq\varepsilon\frac{d}{dt}\langle \dive q, \rho-1\rangle
+\varepsilon\|\dive q\|_{L^2}^2
+\varepsilon \|\dive (q\otimes v)\|_{L^2}^2.
\end{aligned}
\end{equation}
Integrating this over $[0,t]$ and using Young's inequality  yields
\begin{align}
&\quad\|\rho(t)-1\|_{L^2}^2+\varepsilon\int_0^t\|\rho(t')-1\|_{H^1}^2\,dt'
\nonumber\\
&\leq C\big(\|\rho_0-1\|_{L^2}^2+\|\dive q_0\|_{L^2}^2\big)
+C\varepsilon^2 \sup_{t'\in[0,t]}\|\dive q(t')\|_{L^2}^2
\nonumber\\
&\quad+\varepsilon \int_0^t \|\dive q(t')\|_{L^2}^2\,dt'
+C\varepsilon \int_0^t \|\dive (q\otimes v)(t')\|_{L^2}^2\,dt'.\label{rhoH1edd}
\end{align}
Moreover, similarly to Lemma \ref{L2.3}, we have
\begin{align*}
&\quad\frac{d}{dt}\|\nabla \rho\|_{L^2}^2+\eps a^2 \|\nabla \rho\|_{H^1}^2\\
&\leq C \eps\|\rho-1\|_{H^1}^2+\eps\frac{d}{dt}\langle \nabla\dive q,\nabla \rho\rangle+\eps \|\nabla\dive q\|_{L^2}^2+\eps \|\nabla\dive (q\otimes v)\|_{L^2}^2,
\end{align*}
which implies
\begin{align}
&\|\nabla\rho(t)\|_{L^2}^2+\varepsilon\int_0^t\|\nabla\rho(t')\|_{H^1}^2\,dt'
\nonumber\\
&\leq C\eps \int_0^t\|\rho(t')-1\|_{H^1}^2\,dt'+C\big(\|\nabla\rho_0\|_{L^2}^2+\|\nabla\dive q_0\|_{L^2}^2\big)
+C\varepsilon^2 \sup_{t'\in[0,t]}\|\nabla\dive q(t')\|_{L^2}^2
\nonumber\\
&\quad+C\varepsilon\int_0^t\|\nabla\dive q(t')\|_{L^2}^2\,dt'+C\varepsilon \int_0^t \|\nabla\dive (q\otimes v)(t')\|_{L^2}^2\,dt'.\label{nabla2}
\end{align}
Note that the first term on the right-hand side of \eqref{nabla2} has been addressed in \eqref{rhoL2edd} and \eqref{rhoH1edd}. The Moser-type estimate \eqref{F1.1311} leads to 
\begin{align}\label{s}
\|\dive (q\otimes v)\|_{\dot{H}^s}&\le C\|q\|_{L^\infty}\|v\|_{\dot{H}^{s+1}}
+C\|v\|_{L^\infty}\|q\|_{\dot{H}^{s+1}}\nonumber\\
&\le C\|q\|_{H^{m}}\|v\|_{H^{m}},\,\, 0\leq s\leq m-1.
\end{align}
By \eqref{rhoH1edd}-\eqref{s}, we arrive at the second estimate \eqref{rhoL22}.

We now turn to the higher-order estimate \eqref{nablarhok}. In contrast to the proof of Lemma \ref{L2.4}, which proceeds step by step, the argument here is based on a direct energy estimate, with the coefficient
$\|\nabla\phi\|_{L^\infty}^2$ kept in the Gr\"onwall factor. One can apply $\partial^\alpha$ ($|\alpha|=k$, $2\leq k\leq m-1$) to \eqref{4.3}:
\begin{equation}
\begin{aligned}\label{4.3alpha}
\partial_t \partial^\alpha\rho-\eps a^2 \Delta \partial^\alpha\rho+\eps \partial^\alpha \rho= \eps\dive\partial^\alpha\big((\rho-1)\nabla\phi\big)+\eps\partial_t \dive \partial^\alpha q+\eps \dive\dive\partial^\alpha (q\otimes v).
\end{aligned}
\end{equation}

Taking the $L^2$ inner product with $\partial^\alpha\rho$ and applying Young's inequality yields
\begin{align*}
\frac12&\frac{d}{dt}\|\partial^\alpha\rho\|_{L^2}^2
+a^2\eps\|\nabla\partial^\alpha\rho\|_{L^2}^2
+\eps\|\partial^\alpha\rho\|_{L^2}^2\\
&\le \frac{a^2\eps}{4}\|\nabla\partial^\alpha\rho\|_{L^2}^2+\frac{\eps}{4}\|\partial^\alpha\rho\|_{L^2}^2\\
&\quad+\eps\|(\rho-1)\nabla\phi\|_{\dot{H}^k}^2+\eps \|\dive \partial^\alpha(q\otimes v)\|_{L^2}^2+\eps\langle\partial_t \dive \partial^\alpha q, \partial^\alpha\rho\rangle.
\end{align*}
Note the following identity
\begin{equation}\label{pro}
(\rho-1)\nabla\phi
= -\,\dive\!\Big(
\nabla\phi\otimes\nabla\phi
-\frac{1}{2}|\nabla\phi|^2\,\mathrm{Id}
\Big).
\end{equation}
Using the Moser-type estimates \eqref{F1.1311}, it holds 
\begin{align*}
\|(\rho-1)\nabla\phi\|_{\dot{H}^k}
&\le \|\nabla\phi\otimes\nabla\phi\|_{\dot{H}^{k+1}}+\| |\nabla\phi|^2 \|_{\dot{H}^{k+1}}\\
&\leq C\|\nabla\phi\|_{L^\infty}
\|\nabla\phi\|_{\dot{H}^{k+1}}\\
&\leq C \|\nabla\phi\|_{L^\infty}
\|\rho-1\|_{\dot{H}^{k}}.
\end{align*}
By $\partial_t \rho=-\dive q$, one has
\begin{align*}
\eps\langle\partial_t \dive \partial^\alpha q, \partial^\alpha\rho\rangle=\eps\frac{d}{dt}\langle\dive \partial^\alpha q, \partial^\alpha\rho\rangle+\eps\|\partial^\alpha \dive q\|_{L^2}^2.
\end{align*}
Note that the nonlinear convective term has been addressed in \eqref{s}. Consequently, for $2\leq k\leq m-1$, we obtain the differential inequality
\begin{equation}
\begin{aligned}\label{ddthigh}
&\frac{d}{dt}\|\rho-1\|_{\dot{H}^k}^2
+a^2\eps \|\nabla \rho\|_{\dot{H}^k}^2
+\eps\|\rho-1\|_{\dot{H}^k}^2\\
&\le  C\eps\|\nabla\Lambda^{-2} (\rho-1)\|_{L^\infty}^2\|\rho-1\|_{\dot{H}^k}^2+C\eps \|\dive q\|_{\dot{H}^{k}}^2\\
&\quad+\eps\frac{d}{dt}\sum_{|\alpha|=k}\langle\dive \partial^\alpha q, \partial^\alpha\rho\rangle+\eps\|q\|_{H^m}^2  \|v\|_{H^{m}}^2.
\end{aligned}
\end{equation}
Integrating \eqref{ddthigh} in time, noting 
\begin{align*}
\Big|\eps\sum_{|\alpha|=k}\langle\dive \partial^\alpha q, \partial^\alpha\rho\rangle\Big|
\leq \frac{1}{2}\|\rho-1\|_{\dot{H}^k}^2
+C\eps^2\|\dive q\|_{\dot{H}^{k}}^2,
\end{align*}
and using Gr\"onwall's lemma, we end up with the desired bound \eqref{nablarhok}.
\end{proof}

\subsection{Highest-order estimates}

Then, we aim to establish higher-order energy estimates. To this end, it is convenient to introduce the variables
\begin{equation}\label{def:nv}
n:=a\ln\rho,\quad v:=\frac{q}{\rho},\quad n_0:=a\ln \rho_0,\quad v_0:=\frac{q_0}{\rho_0},
\end{equation}
so that \eqref{Euler} can be reformulated into a symmetric system in terms of $(n,v)$:
\begin{equation}\label{2.9}
\begin{cases}
\partial_t n  +v\cdot \nabla n+ a\,\dive v =0,\\
\partial_t v +v\cdot \nabla v + a\,\grad n+\frac{1}{a}\nabla\Lambda^{-2}n +\dfrac{v}{\varepsilon }=-\nabla\Lambda^{-2} \mathcal{G}(n),\\
(n,v)|_{t=0}=(n_0,v_0)
\end{cases}
\end{equation}
with the quadratic term
\begin{align}
\mathcal{G}(n):=e^{\frac{n}{a}}-1-\frac{n}{a}.
\end{align}

\begin{lemma}\label{lemma45}
Assume that \eqref{assumupperlow} holds.  For any $t\in[0,T_\eps)$, it holds  that
\begin{equation}\label{highest}
\begin{aligned}
&\quad\sup_{t'\in[0,t]}\|(n,v)(t')\|_{\dot{H}^m}^2+\frac{1}{\eps}(1-C\eps \sup_{t'\in[0,t]}\|\nabla v(t')\|_{H^{m-1}})\int_0^t \|v(t')\|_{\dot{H}^m}^2\,dt'\\
&\leq C \|(n_0,v_0)\|_{\dot{H}^m}^2 +\frac{C}{\eps}\int_0^t \|v(t')\|_{H^{m-1}}^2\,dt'+C\eps\int_0^t \|n(t')\|_{H^m}^4\,dt'.
\end{aligned}
\end{equation}
\end{lemma}
\begin{proof}
Let $ |\alpha|= m$. 
Applying $\partial^\alpha$ to \eqref{2.9} yields
\begin{equation}\label{sys-nv-alpha}
\left\{
\begin{aligned}
&\partial_t \partial^\alpha n+a\,\dive \partial^\alpha v=-\partial^\alpha(v\cdot\nabla n),\\
&\partial_t \partial^\alpha v+a\,\nabla \partial^\alpha n+\frac1\varepsilon \partial^\alpha v+\frac{1}{a}\nabla\Lambda^{-2}\partial^\alpha n
=-\partial^\alpha((v\cdot\nabla)v)-\nabla\Lambda^{-2}\partial^\alpha \mathcal{G}(n).
\end{aligned}\right.
\end{equation}
Taking the $L^2$ inner product of \eqref{sys-nv-alpha} with $(\partial^\alpha n,\partial^\alpha v)$, we infer
\begin{align}
\frac12\frac{d}{dt}\Big(\|\partial^\alpha n\|_{L^2}^2+\|\partial^\alpha v\|_{L^2}^2\Big)
+\frac1\varepsilon\|\partial^\alpha v\|_{L^2}^2
+\frac{1}{a}\int_{\R^d}\nabla\Lambda^{-2}\partial^\alpha n\cdot \partial^\alpha v\,dx
= -\sum_{i=1,2,3}I^i_\alpha,
\label{Em-alpha-0}
\end{align}
where
\[
I^1_\alpha:=\big\langle \partial^\alpha(v\cdot\nabla n),\,\partial^\alpha n\big\rangle,
\quad
I^2_\alpha:=\big\langle \partial^\alpha((v\cdot\nabla)v),\,\partial^\alpha v\big\rangle,\quad I_\alpha^3:=\big\langle\partial^\alpha\nabla\Lambda^{-2}
\mathcal{G}(n),\partial^\alpha v\big\rangle.\]

We now treat the Poisson term. Using the first equation of \eqref{sys-nv-alpha}, we have
\begin{align}
\int_{\R^d}\nabla\Lambda^{-2}\partial^\alpha n\cdot \partial^\alpha v\,dx
&=-\int_{\R^d}\Lambda^{-2}\partial^\alpha n\,\dive \partial^\alpha v\,dx\nonumber\\
&=\frac1a\int_{\R^d}\Lambda^{-2}\partial^\alpha n\,\partial_t \partial^\alpha n\,dx
+\frac1a\int_{\R^d}\Lambda^{-2}\partial^\alpha n\,\partial^\alpha(v\cdot\nabla n)\,dx\nonumber\\
&=\frac1{2a}\frac{d}{dt}\|\Lambda^{-1}\partial^\alpha n\|_{L^2}^2+a\widetilde I_\alpha,
\label{poisson-alpha}
\end{align}
where
\[
\widetilde I_\alpha:=\frac{1}{a^2}\int_{\R^d}\Lambda^{-2}\partial^\alpha n\,\partial^\alpha(v\cdot\nabla n)\,dx.
\]
Plugging \eqref{poisson-alpha} into \eqref{Em-alpha-0} yields
\begin{align}
\frac12\frac{d}{dt}\Big(\|\partial^\alpha n\|_{L^2}^2+\|\partial^\alpha v\|_{L^2}^2+\frac{1}{a^2}\|\Lambda^{-1}\partial^\alpha n\|_{L^2}^2\Big)
+\frac1\varepsilon\|\partial^\alpha v\|_{L^2}^2
= -\sum_{i=1,2,3}I^i_\alpha-\widetilde I_\alpha.
\label{Em-alpha}
\end{align}
Rewriting the transport terms in commutator forms and integrating by parts, we get
\begin{align*}
I^1_\alpha+I^2_\alpha
&=-\frac12\int_{\R^d}\dive v\,|\partial^\alpha n|^2\,dx
+\big\langle \partial^\alpha(v\cdot\nabla n)-v\cdot\nabla \partial^\alpha n,\,\partial^\alpha n\big\rangle\\
&\quad-\frac12\int_{\R^d}\dive v\,|\partial^\alpha v|^2\,dx
+\big\langle \partial^\alpha (v\cdot\nabla v)-v\cdot\nabla \partial^\alpha v,\,\partial^\alpha v\big\rangle.
\end{align*}
Therefore,
\begin{align}
|I^1_\alpha|+|I^2_\alpha|
&\le C\|\nabla v\|_{L^\infty}\big(\|\partial^\alpha n\|_{L^2}^2+\|\partial^\alpha v\|_{L^2}^2\big)\nonumber\\
&\quad+\|[\partial^\alpha,\,v\cdot\nabla]n\|_{L^2}\,\|\partial^\alpha n\|_{L^2}
+\|[\partial^\alpha,\,v\cdot\nabla]v\|_{L^2}\,\|\partial^\alpha v\|_{L^2}.
\label{IJ-pre}
\end{align}
The classical commutator bound \eqref{F1.14} yields
\begin{equation}\label{comm-est-nv}
\|\partial^\alpha (v\cdot\nabla n)-v\cdot \partial^\alpha \nabla n\|_{L^2}+\|\partial^\alpha(v\cdot\nabla v)-v\cdot\nabla \partial^\alpha v \|_{L^2}
\le C\|\nabla v\|_{H^{m-1}} \|(n,v)\|_{\dot{H}^m}.
\end{equation}
 Plugging \eqref{comm-est-nv} into \eqref{IJ-pre} and using Young's inequality, we obtain
\begin{equation}\label{IJ-est}
|I^1_\alpha|+|I^2_\alpha|
\le  C\|\nabla v\|_{H^{m-1}} \|v\|_{\dot H^m}^2+\frac{1}{2\eps} \|v\|_{\dot{H}^m}^2+\frac{C}{\eps}\|v\|_{H^{m-1}}^2+C\eps\|n\|_{H^m}^4.
\end{equation}
By Taylor's expansion, there exists a smooth function $\widetilde{\mathcal G}$ with
$\widetilde{\mathcal G}(0)=0$ such that
$\mathcal G(n)=\widetilde{\mathcal G}(n)n$. Thus, by \eqref{F1.1311} and
the composition estimate \eqref{F1.15}, we have
\begin{equation}\label{I33}
\begin{aligned}
|I^3_\alpha|&\leq C\|\mathcal{G}(n)\|_{H^{m-1}}\|\partial^{\alpha}v\|_{L^2}\\
&\leq C\|\tilde{\mathcal{G}}(n)\|_{H^{m-1}}\|n\|_{H^{m-1}}\|\partial^{\alpha}v\|_{L^2}\\
&\leq \frac{1}{4\eps} \|\partial^{\alpha}v\|_{L^2}^2+\eps \|n\|_{H^{m-1}}^4.
\end{aligned}
\end{equation}
It remains to estimate $\widetilde I_\alpha$. Moser-type product law \eqref{F1.1311} leads to
\begin{equation}
\begin{aligned}
|\widetilde I_\alpha|
&\le \frac{1}{a^2}\|\Lambda^{-1} \partial^{\alpha} n\|_{L^2} \|\Lambda^{-1}\partial^{\alpha}(v\cdot\nabla n)\|_{L^2}\\
&\leq C\|n\|_{H^{m-1}}(\|v\|_{L^\infty} \|\nabla n\|_{H^{m-1}}+\|v\|_{H^{m-1}}\|\nabla n\|_{L^\infty}) \\
&\leq  \frac{C}{\eps}\|v\|_{H^{m-1}}^2+ \eps \|n\|_{H^{m-1}}^2\|n\|_{H^{m}}^2.
\label{Itilde1-est0}
\end{aligned}
\end{equation}

Gathering \eqref{IJ-est}, \eqref{I33} and \eqref{Itilde1-est0} in \eqref{Em-alpha}, and summing over all $|\alpha|=m$, we obtain
\begin{equation}
\begin{aligned}
&\quad\frac{d}{dt}\sum_{|\alpha|=m}\Big(\|\partial^\alpha n\|_{L^2}^2+\|\partial^\alpha v\|_{L^2}^2+\frac{1}{a^2}\|\Lambda^{-1}\partial^\alpha n\|_{L^2}^2\Big)
+\Big(\frac1\varepsilon-C \|\nabla v\|_{H^{m-1}}\Big)\sum_{|\alpha|=m}\|\partial^\alpha v\|_{L^2}^2\\
&\le \frac{C}{\eps}\|v\|_{H^{m-1}}^2+C\eps\|n\|_{H^m}^4.
\label{Em-top}
\end{aligned}
\end{equation}
Integrating \eqref{Em-top} in time yields \eqref{highest}.
\end{proof}

\section{Error estimates and upper and lower bounds for \texorpdfstring{$\rho$}{rho}}\label{section:EP2}
\vspace{2mm}

Since the limiting system admits a global classical solution $\rho^*$ with uniform upper and lower bounds, the goal of this section is to derive quantitative estimates for the error $\rho^\eps-\rho^*$. 

To this end, we introduce the rescaled variables (see \eqref{scaling}):
\begin{align}\label{4.61}
\rho^\eps(t,x)=\rho\Big(\frac{t}{\eps},x\Big),\quad 
q^\eps(t,x)=\frac{1}{\eps}q\Big(\frac{t}{\eps},x\Big),\quad v^\eps=\frac{q^\eps}{\rho^\eps}.
\end{align}
Let the initial layer correction $Z_L^\eps$ be given by \eqref{qL}. Inspired by \cite{CBPSX,pengJFA}, we also introduce the effective variable associated with Darcy's law:
\begin{align}
Z^\eps=q^\eps+a^2\nabla\rho^\eps+\rho^\eps \nabla \Lambda^{-2}(\rho^\eps-1)-Z_L^\eps,
\end{align}
with a zero initial value due to the definition of $Z_L^\eps$.

Then $\rho^\eps$ satisfies
\begin{align}\label{ddwide}
\partial_t \rho^\eps-a^2 \Delta \rho^\eps
-\dive\big(\rho^\eps \nabla \Lambda^{-2}(\rho^\eps-1)\big)
=-\dive Z^\eps-\dive Z_L^\eps. 
\end{align}
Moreover, $Z_L^\eps$ can be written in the form \begin{align}\label{Zdecom}
Z_L^\eps=q_L^\eps+q_{L,1}^\eps\quad\text{with}\quad 
q_L^\eps=e^{-\frac{t}{\eps^2}}\frac{1}{\eps}q_0 \quad\text{and}\quad 
q_{L,1}^\eps=e^{-\frac{t}{\eps^2}}
\bigl(a^2\nabla\rho_0+\rho_0\nabla\phi_0\bigr).
\end{align}
Here, the $L^2(\mathbb{R}_+;H^m)$-norm of $q_L^\eps$ is uniformly bounded but does not vanish as $\eps\to0$, while its $L^1(\mathbb{R}_+;H^m)$-norm is of order $\mathcal{O}(\eps)$:
\begin{align}
\int_0^\infty \|q_{L}^\eps(t')\|_{H^m}^2\,dt'
&\leq  \int_0^\infty e^{-\frac{2t'}{\eps^2}}\frac{1}{\eps^2}\,dt' \|q_0 \|_{H^m}^2
\leq  C \|q_0 \|_{H^m}^2,\label{tlldeqL11}\\
\int_0^\infty \|q_{L}^\eps(t')\|_{H^m}\,dt'
&\leq \int_0^\infty e^{-\frac{t'}{\eps^2}}\frac{1}{\eps}\,dt'\,\|q_0 \|_{H^m}
\leq  C\eps \|q_0 \|_{H^m}\label{tlldeqL1}.
\end{align}
On the other hand, $q_{L,1}^\eps$ has a convergence rate in a lower-order space: 
\begin{align}\label{tlldeqL2}
\int_0^\infty \|q_{L,1}^\eps(t')\|_{H^{m-1}}^2\,dt'&\leq C\eps^2(\|\nabla \rho_0 \|_{H^{m-1}}^2+\|\rho_0\nabla\phi_0\|_{H^{m-1}}^2)\nonumber\\
&\leq C\eps^2(1+\|\rho_0-1\|_{H^{m-1}}^2)(\|\nabla\phi_0\|_{L^2}^2+\|\rho_0-1\|_{H^{m}}^2).
\end{align}
Here we have used
\[  \|\rho_0\nabla\phi_0\|_{H^{m-1}}\le C(1+\|\rho_0-1\|_{H^{m-1}})\|\nabla\phi_0\|_{H^{m-1}} \]
and
\[ \|\nabla\phi_0\|_{H^{m-1}}\le C\big(\|\nabla\phi_0\|_{L^2}+\|\rho_0-1\|_{H^{m-2}}\big), \]
as $\nabla\phi_0=\nabla\Lambda^{-2}(\rho_0-1)$.

\begin{lemma}\label{lemma4555}
Assume that \eqref{assumupperlow} holds.  For all $t\in [0, T_\eps)$, it holds
\begin{equation}
\begin{aligned}\label{4.65}
&\quad \sup_{t'\in [0,\eps t]}\|\rho^\eps-\rho^*\|_{H^{m-1}\cap\dot{H}^{-1}}^2
+\int_0^{\eps t} \|\rho^\eps-\rho^*\|_{H^{m}\cap\dot{H}^{-1}}^2\,dt'\\
&\leq Ce^{C\int_0^\infty\|\rho^*-1\|_{H^{m-1}}^2\,dt'+C\int_0^{\eps t}\|\nabla \rho^\eps\|_{H^{m-1}}^2\,dt'}\\
&\quad\times
\Big( \frac{1}{\eps^2}\int_0^{\eps t} \|Z^\eps\|_{H^{m-1}}^2\,dt'+ \|q_0 \|_{H^m}^2+(1+\|\rho_0-1\|_{H^{m-1}}^2)(\|\nabla\phi_0\|_{L^2}^2+\|\rho_0-1\|_{H^{m}}^2)\Big)\eps^2.
\end{aligned}
\end{equation}
\end{lemma}

\begin{proof}
Let $\delta \rho^\eps=\rho^\eps-\rho^*$. Then, subtracting the limiting equation yields
\begin{equation}
\begin{aligned}\label{detlarho}
&\quad\partial_t \delta \rho^\eps-a^2 \Delta \delta \rho^\eps\\
&=\dive \big( \rho^\eps \nabla \Lambda^{-2}\delta \rho^\eps\big)
+\dive \big( \delta \rho^\eps\nabla \Lambda^{-2}(\rho^*-1)\big)
-\dive Z^\eps-\dive q_{L}^\eps-\dive q_{L,1}^\eps.
\end{aligned}
\end{equation}

The first step is to establish the $L^2$ estimate. Taking the $L^2$ inner product with $\delta \rho^\eps$, we obtain
\begin{align}
\frac{1}{2}\frac{d}{dt}\|\delta \rho^\eps\|_{L^2}^2+a^2\|\nabla \delta \rho^\eps\|_{L^2}^2
&=\int_{\mathbb{R}^d} \dive ( \rho^\eps \nabla \Lambda^{-2}\delta \rho^\eps)\delta \rho^\eps\,dx\nonumber\\
&\quad+\int_{\mathbb{R}^d} \dive \big( \delta \rho^\eps\nabla \Lambda^{-2}(\rho^*-1)\big) \delta \rho^\eps\,dx\nonumber\\
&\quad-\langle \dive Z^\eps,\delta \rho^\eps\rangle
-\langle \dive Z_L^\eps,\delta \rho^\eps\rangle.
\end{align}
For the first term, it follows that
\begin{align*}
\int_{\R^d}\dive(\rho^\eps\nabla\Lambda^{-2}\delta\rho^\eps)\,\delta\rho^\eps\,dx
&=-\int_{\R^d}\rho^\eps\nabla\Lambda^{-2}\delta\rho^\eps\cdot\nabla\delta\rho^\eps\,dx\nonumber\\
&=\int_{\R^d}\rho^\eps(\Delta\Lambda^{-2}\delta\rho^\eps)\,\delta\rho^\eps\,dx
+\int_{\R^d}\nabla\rho^\eps\cdot\nabla\Lambda^{-2}\delta\rho^\eps\,\delta\rho^\eps\,dx\nonumber\\
&=-\int_{\R^d}\rho^\eps|\delta\rho^\eps|^2\,dx
+\int_{\R^d}\nabla\rho^\eps\cdot\nabla\Lambda^{-2}\delta\rho^\eps\,\delta\rho^\eps\,dx.
\end{align*}
Here,
\begin{align*}
\left|\int_{\mathbb{R}^d} \nabla \rho^\eps \cdot \nabla \Lambda^{-2}\delta \rho^\eps  \delta \rho^\eps\,dx\right|
&\le C\|\nabla \rho^\eps\|_{H^{m-1}} 
\big(\|\delta \rho^\eps\|_{L^2}^2+\|\delta \rho^\eps\|_{\dot{H}^{-1}}^2\big)\\
&\leq \frac{\tilde{\rho}_1}{8} (\|\delta \rho^\eps\|_{L^2}^2+\|\nabla\Lambda^{-2}\delta\rho^\eps\|_{L^2}^2)+C\|\nabla \rho^\eps\|_{H^{m-1}}^2 (\|\delta \rho^\eps\|_{L^2}^2+\|\delta\rho^\eps\|_{\dot{H}^{-1}}^2).
\end{align*}
For the second nonlinear term,
\begin{align*}
\left|\int_{\mathbb{R}^d} \dive \big( \delta \rho^\eps\nabla \Lambda^{-2}(\rho^*-1)\big) \delta \rho^\eps\,dx\right|
&\le C\|\rho^*-1\|_{H^{m-1}}\|\delta\rho^\eps\|_{L^2}^2\\
&\le\frac{\tilde{\rho}_1}{8} \|\delta \rho^\eps\|_{L^2}^2+C\|\rho^*-1\|_{H^{m-1}}^2\|\delta\rho^\eps\|_{L^2}^2.
\end{align*}
For the remainder terms,
\begin{align*}
|\langle \dive Z^\eps,\delta \rho^\eps\rangle|
&\le \frac{a^2}{4}\|\nabla \delta \rho^\eps\|_{L^2}^2+C\|Z^\eps\|_{L^2}^2,\\
|\langle \dive q_{L}^\eps,\delta \rho^\eps\rangle|
&\le \|\dive q_{L}^\eps\|_{L^2}\|\delta \rho^\eps\|_{L^2},\\
|\langle \dive q_{L,1}^\eps,\delta \rho^\eps\rangle|
&\le \frac{a^2}{4}\|\nabla \delta \rho^\eps\|_{L^2}^2+C\| q_{L,1}^\eps\|_{L^2}^2.
\end{align*}
Collecting the above estimates yields
\begin{equation}
\begin{aligned}\label{L2-delta}
&\quad \frac{d}{dt}\|\delta \rho^\eps\|_{L^2}^2+a^2\|\nabla \delta \rho^\eps\|_{L^2}^2+\tilde{\rho}_1\|\delta \rho^\eps\|_{L^2}^2 \\
&\le C(\|\rho^*-1\|_{H^{m-1}}^2+\|\nabla \rho^\eps\|_{H^{m-1}}^2)
(\|\delta \rho^\eps\|_{L^2}^2+\|\delta \rho^\eps\|_{\dot{H}^{-1}}^2)\\
&\quad+C\|Z^\eps\|_{L^2}^2+C\|q_{L,1}^\eps\|_{L^2}^2+\|\dive q_{L}^\eps\|_{L^2}\|\delta \rho^\eps\|_{L^2}.
\end{aligned}
\end{equation}
To address the right-hand side of \eqref{L2-delta}, one needs to derive the $\dot{H}^{-1}$ estimate. Taking the inner product of \eqref{detlarho} with $\Lambda^{-2}\delta \rho^\eps$, we get
\begin{equation}
\begin{aligned}\label{4.73}
&\quad\frac{1}{2}\frac{d}{dt}\|\Lambda^{-1}\delta\rho^\eps\|_{L^2}^2
+a^2\|\nabla\Lambda^{-1}\delta\rho^\eps\|_{L^2}^2+\int_{\mathbb{R}^d} \rho^\eps |\nabla\Lambda^{-2}\delta\rho^\eps|^2\,dx\\
&\le C\|\rho^*-1\|_{H^{m-1}} \|\delta \rho^\eps\|_{H^{-1}}^2
+C\|Z^\eps\|_{L^2}^2+C\|q_{L,1}^\eps\|_{L^2}^2+C\|\dive q_{L}^\eps\|_{L^2}\|\Lambda^{-1}\delta \rho^\eps\|_{L^2}.
\end{aligned}
\end{equation}

Then, we aim to establish higher-order estimates. Applying $\partial^\alpha$ with $|\alpha|= m-1$ to \eqref{detlarho}, taking the $L^2$ inner product with $\partial^\alpha \delta \rho^\eps$, and using Moser estimates yields
\begin{equation}\label{4.74}
\begin{aligned}
\frac{d}{dt}\|\delta \rho^\eps\|_{\dot{H}^{m-1}}^2
+a^2\|\delta \rho^\eps\|_{\dot{H}^{m}}^2
&\le C(\|\rho^*-1\|_{H^{m-1}}^2+\|\nabla \rho^\eps\|_{H^{m-1}}^2)
(\|\delta \rho^\eps\|_{H^{m-1}}^2+\|\delta \rho^\eps\|_{\dot{H}^{-1}}^2)\\
&\quad
+C\|Z^\eps\|_{H^{m-1}}^2+C\|q_{L,1}^\eps\|_{H^{m-1}}^2+\|\dive q_{L}^\eps\|_{H^{m-1}}\|\delta \rho^\eps\|_{\dot{H}^{m-1}}.
\end{aligned}
\end{equation}

Combining the above estimates \eqref{L2-delta}, \eqref{4.73} and \eqref{4.74}, we have
\begin{equation*}
\begin{aligned}
&\sup_{t'\in [0,t]}\|\delta \rho^\eps(t')\|_{H^{m-1}\cap \dot{H}^{-1}}^2
+\int_0^t \|\delta \rho^\eps(t')\|_{H^{m}\cap \dot{H}^{-1}}^2\,dt'\nonumber\\
&\leq C\int_0^t(\|\rho^*(t')-1\|_{H^{m-1}}^2+\|\nabla \rho^\eps(t')\|_{H^{m-1}}^2)\|\delta \rho^\eps(t')\|_{H^{m-1}\cap \dot{H}^{-1}}^2\,dt'\\
&\quad+C \int_0^t (\|Z^\eps(t')\|_{H^{m-1}}^2+\|q_{L,1}^\eps(t')\|_{H^{m-1}}^2)\,dt'+\int_0^t \|q_{L}^\eps(t')\|_{H^m}\,dt' \sup_{t'\in [0,t]}\|\delta \rho^\eps(t')\|_{H^{m-1}\cap \dot{H}^{-1}}\\
&\leq \frac{1}{2}\sup_{t'\in [0,t]} \|\delta \rho^\eps(t')\|_{H^{m-1}\cap \dot{H}^{-1}}^2+C\int_0^t(\|\rho^*(t')-1\|_{H^{m-1}}^2+\|\nabla \rho^\eps(t')\|_{H^{m-1}}^2)\|\delta \rho^\eps(t')\|_{H^{m-1}\cap \dot{H}^{-1}}^2\,dt'\\
&\quad+C \int_0^t (\|Z^\eps(t')\|_{H^{m-1}}^2+\|q_{L,1}^\eps(t')\|_{H^{m-1}}^2)\,dt'+C\Big( \int_0^t \|q_{L}^\eps(t')\|_{H^m}\,dt'\Big)^2,
\end{aligned}
\end{equation*}
which, together with Gr\"onwall's inequality, \eqref{tlldeqL1} and \eqref{tlldeqL2}, concludes \eqref{4.65}.
\end{proof}

To derive the convergence rate, as indicated in \eqref{4.65}, it remains to obtain an $\mathcal{O}(\eps)$ bound of $Z^\eps$ in $L^2(0,\eps T;H^{m-1})$. 
To this end, we rewrite the equation for $ q^\eps- q_L^\eps$ as
\begin{equation}\label{error10000}
\eps^2\partial_t (q^\eps-q_L^\eps)+(q^\eps-q_L^\eps)
+a^2\nabla\rho^\eps+\rho^\eps \nabla \Lambda^{-2}(\rho^\eps-1)
=-\eps^2 \dive (q^\eps\otimes v^\eps).
\end{equation}
Recall that
\[ Z^\eps=q^\eps+a^2\nabla\rho^\eps+\rho^\eps\nabla\Lambda^{-2}(\rho^\eps-1)-Z_L^\eps. \]
Differentiating $Z^\eps$ in time and using \eqref{error10000} yields
\begin{equation}\label{eq:Z-eq}
\partial_t Z^\eps+\frac{1}{\varepsilon^2} Z^\eps
=a^2 \nabla\partial_t \rho^\eps+\tilde{f}^\eps,
\end{equation}
with $Z^\eps|_{t=0}=0$ and
\begin{equation}\label{def:f-tilde}
\tilde{f}^\eps
:=\partial_t\big(\rho^\eps \nabla \Lambda^{-2}(\rho^\eps-1)\big)
-\dive\!(q^\eps\otimes v^\eps).
\end{equation}

\begin{lemma}\label{lem2.5}
Assume that \eqref{assumupperlow} holds. Then, for all $t\in[0,T_\varepsilon)$, it holds
\begin{equation}\label{2.17}
\begin{aligned}
&\quad \sup_{t'\in[0,\eps t]}\|Z^\eps(t')\|_{H^{m-1}}^2
+\frac{1}{\eps^2}\int_0^{\eps t}\|Z^\eps(t')\|_{H^{m-1}}^2\,dt'\\
&\leq C\|q_0 \|_{H^m}^2
+C(1+\|\rho_0-1\|_{H^{m-1}}^2)
(\|\nabla\phi_0\|_{L^2}^2+\|\rho_0-1\|_{H^{m}}^2)\\
&\quad+C\sup_{t'\in[0,\eps t]}\|(\rho^\eps-1)(t')\|_{H^{m}}^2\\
&\quad+ C \int_0^{\eps t}
\big(1+\|\rho^\eps(t')-1\|_{H^m}^2+\eps^2\| v^\eps(t')\|_{H^m}^2\big)
(\| q^\eps(t')\|_{H^m}^2+\|\rho^\eps(t')-1\|_{H^m}^2)\,dt'.
\end{aligned}
\end{equation}
\end{lemma}

\begin{proof}
Applying $\partial^\alpha$ with $|\alpha|\le m-1$ to \eqref{eq:Z-eq}, taking the $L^2$ inner product with $\partial^\alpha Z^\eps$, summing over $|\alpha|\le m-1$ and noting $Z^\eps(0)=0$, we obtain
\begin{equation}
\begin{aligned}\label{Z-energy-raw}
&\quad\frac12\|Z^\eps(t)\|_{H^{m-1}}^2+\frac{1}{\eps^2}\int_0^t\|Z^\eps(t')\|_{H^{m-1}}^2\,dt'\\
&=\int_0^t\big\langle a^2\nabla\partial_t\rho^\eps,Z^\eps\big\rangle_{H^{m-1}}\,dt'
+\int_0^t\big\langle \tilde f^\eps,Z^\eps\big\rangle_{H^{m-1}}\,dt'.
\end{aligned}
\end{equation}
We first handle the coupling term. The definition of $Z^\eps$ implies
\begin{align*}
\int_0^t\big\langle a^2\nabla\partial_t\rho^\eps,Z^\eps\big\rangle_{H^{m-1}}\,dt'
&=\int_0^t\big\langle a^2\nabla\partial_t\rho^\eps,\,a^2\nabla\rho^\eps\big\rangle_{H^{m-1}}\,dt'-\int_0^t\langle a^2\nabla\partial_t\rho^\eps, q_{L,1}^\eps\rangle_{H^{m-1}}\,dt' \\
&\quad+\int_0^t\big\langle a^2\nabla\partial_t\rho^\eps,\, q^\eps- q_L^\eps+\rho^\eps\nabla\Lambda^{-2}(\rho^\eps-1)\big\rangle_{H^{m-1}}\,dt'.
\end{align*}
The first term satisfies
\begin{equation*}
\int_0^t\big\langle a^2\nabla\partial_t\rho^\eps,\,a^2\nabla\rho^\eps\big\rangle_{H^{m-1}}\,dt'
=\frac{a^4}{2}\int_0^t\frac{d}{dt'}\|\nabla\rho^\eps(t')\|_{H^{m-1}}^2\,dt'\leq\frac{a^4}{2} \|\nabla\rho^\eps(t)\|_{H^{m-1}}^2 .
\end{equation*}
For the second term, we have
\begin{align*}
&\quad\int_0^t \langle a^2\nabla\partial_t\rho^\eps, q_{L,1}^\eps\rangle_{H^{m-1}}\,dt' \\
&\leq \langle a^2\nabla \rho^\eps, q_{L,1}^\eps\rangle_{H^{m-1}}\Big|^{t}_{0}\\
&\quad+\int_0^t\frac{1}{\eps^2} e^{-\frac{t'}{\eps^2}}\,dt'\,  a^2\sup_{t'\in[0,t]}\| \nabla\rho^\eps\|_{H^{m-1}} \| a^2\nabla\rho_0+\rho_0\nabla\phi_0\|_{H^{m-1}}\\
&\leq C\sup_{t'\in[0,t]}\| \nabla\rho^\eps\|_{H^{m-1}}^2+C(1+\|\rho_0-1\|_{H^{m-1}}^2)(\|\nabla\phi_0\|_{L^2}^2+\|\rho_0-1\|_{H^{m}}^2).
\end{align*}

For the remainder term, using $\partial_t\rho^\eps=-\dive q^\eps$ and integrating by parts, we infer
\begin{align*}
&\quad\big|\big\langle a^2\nabla\partial_t\rho^\eps,\, q^\eps- q_L^\eps+\rho^\eps\nabla\Lambda^{-2}(\rho^\eps-1)\big\rangle_{H^{m-1}}\big|\\
&\le C\| q^\eps\|_{H^m}^2+C\| q_L^\eps\|_{H^{m}}^2
+C
\|\dive(\rho^\eps\nabla\Lambda^{-2}(\rho^\eps-1))\|_{H^{m-1}}^2.
\end{align*}

For the second term on the right-hand side of \eqref{Z-energy-raw},  using Young's inequality, we get
\begin{equation*}
\Big|\int_0^t\langle \tilde f^\eps,Z^\eps\rangle_{H^{m-1}}\,dt'\Big|
\le \frac{1}{2\eps^2}\int_0^t\|Z^\eps(t')\|_{H^{m-1}}^2\,dt'
+C\eps^2\int_0^t\|\tilde f^\eps(t')\|_{H^{m-1}}^2\,dt'.
\end{equation*}
By \eqref{pro} and Moser-type estimates,
\begin{align*}
&\quad\|\dive(\rho^\eps\nabla\Lambda^{-2}(\rho^\eps-1))\|_{H^{m-1}}\\
&\le C\| \rho^\eps-1\|_{H^{m-1}}+C\Big(\|\rho^\eps-1\|_{L^\infty}\|\rho^\eps-1\|_{H^m}
+\|\nabla\Lambda^{-2}(\rho^\eps-1)\|_{L^\infty}\|\rho^\eps-1\|_{H^m}\Big)\nonumber\\
&\le C(1+\|\rho^\eps-1\|_{H^m})\|\rho^\eps-1\|_{H^m}.
\end{align*}
Thus,
\begin{equation}\label{coupling-final}
\begin{aligned}
&\quad
\left|
\big\langle a^2\nabla\partial_t\rho^\eps,\,
q^\eps-q_L^\eps+\rho^\eps\nabla\Lambda^{-2}(\rho^\eps-1)
\big\rangle_{H^{m-1}}
\right|\\
&\le C\|q^\eps\|_{H^m}^2+C\|q_L^\eps\|_{H^{m}}^2
+C(1+\|\rho^\eps-1\|_{H^m}^2)\|\rho^\eps-1\|_{H^m}^2.
\end{aligned}
\end{equation}
Using $\partial_t\rho^\eps=-\dive q^\eps$ and the Moser-type estimate \eqref{F1.1311}, we have
\begin{equation}\label{quad-flux-est}
\|\dive (q^\eps\otimes v^\eps)\|_{H^{m-1}}
\le C\| q^\eps\|_{H^m}\| v^\eps\|_{H^m}.
\end{equation}
Therefore,
\begin{equation}\label{f-tilde-final}
\|\tilde f^\eps\|_{H^{m-1}}^2
\le C(1+\|\rho^\eps-1\|_{H^m}^2)\|q^\eps\|_{H^m}^2
+C\|q^\eps\|_{H^m}^2\|v^\eps\|_{H^m}^2.
\end{equation}

Combining the above estimates with \eqref{Z-energy-raw}, using \eqref{tlldeqL11} and \eqref{tlldeqL2}, integrating the resulting inequality in time and recalling $Z^\eps(0)=0$, we get \eqref{2.17}.
\end{proof}

Then, we arrive at the desired error estimate and obtain control of the lower and upper bounds for $\rho$. 
\begin{lemma}\label{lem2.6}
Assume that \eqref{assumupperlow} holds. Then, for all $t\in[0,T_\eps)$, we have
\begin{equation}\label{2.21}
\sup_{t'\in[0,t]}\|\rho(t')-\rho^*(\eps t')\|_{H^{m-1}\cap\dot{H}^{-1}}^2
+\varepsilon \int_0^t\|\rho(t')-\rho^*(\eps t')\|_{H^m\cap\dot{H}^{-1}}^2\,dt'
\le C\varepsilon ^2 B(t),
\end{equation}
and
\begin{equation}\label{2.22}
|\rho(t,x)-\rho^*(\eps t,x)|\le
C\varepsilon \sqrt{B(t)},\quad (t,x)\in [0,T_\eps)\times\mathbb{R}^d.
\end{equation}
where $\rho_1$ and $\rho_2$ are, respectively, the lower and upper bounds of $\rho^*$ (see \eqref{F2.4}), and  $B(t)$ is given by
\begin{equation}\label{Bt}
\begin{aligned}
B(t):&=Ce^{C \int_0^\infty\|\rho^*(t')-1\|_{H^{m-1}}^2\,dt'+C\eps \int_0^{t}\|\nabla \rho(t')\|_{H^{m-1}}^2\,dt'}\\
&\quad\times \Big( \|q_0 \|_{H^m}^2+(1+\|\rho_0-1\|_{H^{m-1}}^2)(\|\nabla\phi_0\|_{L^2}^2+\|\rho_0-1\|_{H^{m}}^2)\\
&\quad\quad+\sup_{t'\in[0,t]}\|\rho(t')-1\|_{H^{m}}^2+ (1+\sup_{t'\in[0,t]}\|v(t')\|_{H^m}^2)\frac{1}{\eps}\int_0^{t} \|q(t')\|_{H^m}^2\,dt'\\
&\quad\quad+(1+\sup_{t'\in[0,t]}\|\rho(t')-1\|_{H^m}^2)\eps\int_0^{t}\|\rho(t')-1\|_{H^m}^2\,dt'\Big).
\end{aligned}
\end{equation}
\end{lemma}

\begin{proof}
From the transforms \eqref{4.61}, one readily sees that
\begin{align*}
\sup_{t'\in [0,\eps t]}\|(\rho^\eps-1)(t')\|_{H^m}^2&=\sup_{t'\in [0,t]}\|(\rho-1)(t')\|_{H^m}^2,\\
\int_0^{\eps t} \|(\rho^\eps-1)(t')\|_{H^m}^2\,dt'&=\eps \int_0^{t} \|(\rho-1)(t')\|_{H^m}^2\,dt',\\
\eps^2\sup_{t'\in [0,\eps t]}\|v^\eps(t')\|_{H^m}^2&=\sup_{t'\in [0,t]}\|v(t')\|_{H^m}^2,\\
\int_0^{\eps t} \| q^\eps(t')\|_{H^m}^2\,dt'&=\frac{1}{\eps}\int_0^{t} \|q(t')\|_{H^m}^2\,dt',\\
\sup_{t'\in [0, \eps t]}\|(\rho^\eps-\rho^*)(t')\|_{H^{m-1}\cap\dot{H}^{-1}}^2&=\sup_{t'\in [0,  t]}\|\rho(t')-\rho^*(\eps t')\|_{H^{m-1}\cap\dot{H}^{-1}}^2,\\
\int_0^{\eps t} \|(\rho^\eps-\rho^*)(t')\|_{H^{m}\cap\dot{H}^{-1}}^2\,dt'&=\eps \int_0^{t} \|\rho(t')-\rho^*(\eps t')\|_{H^{m}\cap\dot{H}^{-1}}^2\,dt'.
\end{align*}
Therefore, with the estimates \eqref{4.65} and \eqref{2.17}, we obtain \eqref{2.21}. By the Sobolev embedding $H^{m-1}\hookrightarrow L^\infty$, this yields \eqref{2.22}.
\end{proof}


\section{Proof of Theorem 1.1}\label{section:EP3}
\subsection{Bootstrap}
For fixed $0<\eps<1$, one can refer to, e.g., \cite{11,Y1} for the local existence result of a solution $(\rho,q,\phi)$ to the Cauchy problem \eqref{Euler}--\eqref{Eulerd}. Let 
$T_\eps$ be the maximal time of local existence. For  $T\in[0,T_\varepsilon)$, we denote by $C_i$ positive constants depending only on $(\rho_1,\rho_2)$, $d$, $a$ and $m$, but independent of $T$, $\varepsilon$ and $E_0$. We introduce
\[
E_T=\sup_{t'\in[0,T]}\big(\|\rho(t')-1\|_{H^m}+\|\nabla\phi(t')\|_{L^{2}}+\|q(t')\|_{H^m}\big)
+\left(\frac{1}{\varepsilon }\int_0^T \|q(t')\|_{H^m}^2\,dt'\right)^{1/2}.
\]
We aim to justify that if
\begin{align}\label{smallness-main}
e^{\bar{C}e^{\bar{C}E_0^2}}E_0\eps\leq \delta_0
\end{align}
for some sufficiently large constant $\bar{C}>0$ and some sufficiently small constant $\delta_0>0$ independent of $E_0$ and $\eps$, and if
\begin{align}
&E_T\leq C_* E_0,\label{apriori}\\
&\tilde{\rho}_1\leq \rho(t,x) \leq \tilde{\rho}_2,\quad (t,x)\in [0,T]\times \mathbb{R}^d,\label{aprioriupper}
\end{align}
then we have
\begin{align}
&E_T\leq \frac{1}{2}C_* E_0,\label{apriori:e}\\
&\frac{5}{4}\tilde{\rho}_1\leq \rho(t,x) \leq \frac{3}{4}\tilde{\rho}_2,\quad (t,x)\in [0,T]\times \mathbb{R}^d.\label{lowupper}
\end{align}
Here, we set
\begin{align}\label{tilderho12}
\tilde{\rho}_1:=\frac{2}{5}\rho_1,\quad \tilde{\rho}_2:=2\rho_2.
\end{align}
The constants $\delta_0>0$ and $C_*>1$ will be fixed below. Consequently, we can apply a bootstrap argument to prove $T_\eps=\infty$.

 The proof of \eqref{apriori:e}-\eqref{lowupper} involves several levels of energy-dissipation estimates. The analysis relies strongly on the damping dissipation of $q$ and the limiting solution of the drift--diffusion system.

 \begin{itemize}
     \item {\emph{Step 1: $H^{m-1}$-energy estimates.}}
 \end{itemize}

 Under \eqref{apriori}--\eqref{aprioriupper}, we have
\begin{align}
\|v\|_{H^m}\leq C_1(1+\|\rho-1\|_{H^m})\|q\|_{H^m}.
\end{align}
Consequently, it holds
\begin{align}
&\quad \varepsilon\int_0^T \|q(t')\|_{H^m}^2 \|v(t')\|_{H^m}^2 \,dt'\nonumber\\
&\leq C_1(1+\sup_{t'\in[0,T]}\|\rho(t')-1\|_{H^m})^2
\sup_{t'\in[0,T]}\|q(t')\|_{H^m}^2
\varepsilon\int_0^T \|q(t')\|_{H^m}^2\,dt'\nonumber\\
&\leq C_1(1+C_*E_0)^2 C_*^4 E_0^4 \eps^2.\label{q2}
\end{align}
Then, it follows from \eqref{L2}, \eqref{Hk-energy-int} and \eqref{q2} that
\begin{equation}\label{energym-1}
\begin{aligned}
&\quad\sup_{t\in[0,T]}\|(\rho-1,q)(t)\|_{H^{m-1}}^2
+\frac{1}{\eps}\int_0^T \|q(t')\|_{H^{m-1}}^2\,dt'\\
&\leq C_2\Big(\|(\rho_0-1,q_0)\|_{H^{m-1}}^2
+(1+C_*E_0)^2 C_*^4 E_0^4 \eps^2\\
&\quad+\eps \int_0^T \|\nabla\phi(t')\|_{L^\infty}^2
\|\rho(t')-1\|_{H^{m-1}}^2\,dt'\Big).
\end{aligned}
\end{equation}

 \begin{itemize}
     \item {\emph{Step 2:  $L^2(0,T;L^\infty)$-estimate of $\nabla\phi$ and dissipation estimate of $\rho$.}}
 \end{itemize}

To analyze the last term on the right-hand side of \eqref{energym-1}, the key ingredient is to establish a uniform weighted $L^2(0,T;L^\infty)$-estimate for $\nabla\phi$. For $d\le 5$, in view of the Gagliardo--Nirenberg inequality (Lemma \ref{Lt9}), we have
\begin{align*}
\|\nabla\phi\|_{L^\infty}^2
\le C_3\|\nabla\phi\|_{L^2}^{2-\frac d3}
\|\nabla\phi\|_{\dot{H}^3}^{\frac d3}
\le C_3\big(\|\nabla\phi\|_{L^2}^2+\|\rho-1\|_{\dot{H}^2}^2\big).
\end{align*}
According to the dissipation estimates \eqref{rhoL21} and \eqref{rhoL22} for $\rho$, the bootstrap assumption \eqref{apriori}, and \eqref{q2}, one has
\begin{equation*}
\begin{aligned}
&\quad \varepsilon \int_0^T\|\nabla\phi(t')\|_{L^\infty}^2\,dt'\\
&\leq C_4 \big(\|\rho_0-1\|_{L^2}^2+\|\nabla\phi_0\|_{L^2}^2+\|q_0\|_{L^2}^2\big)\\
&\quad
+C_4\varepsilon^2 \Big(\sup_{t'\in[0,T]}\|\dive q(t')\|_{L^2}^2
+\frac{1}{\varepsilon} \int_0^T \|\dive q(t')\|_{L^2}^2\,dt'\Big)\\
&\quad+C_4\varepsilon\int_0^T \|q(t')\|_{H^{m}}^2 \|v(t')\|_{H^m}^2 \,dt' \\
&\leq C_4 \big(E_0^2+C_*^2E_0^2\eps^2+(1+C_*E_0)^2 C_*^4 E_0^4\eps^2\big).
\end{aligned}
\end{equation*}
We now impose the following smallness condition:
\begin{align}\label{delta1}
C_*E_0\eps\leq 1,
\qquad
(1+C_*E_0)C_*^2E_0\eps\leq 1.
\end{align}
Hence
\begin{equation}\label{LinftyEPs}
\varepsilon \int_0^T\|\nabla\phi(t')\|_{L^\infty}^2\,dt'
\leq C_5(1+E_0^2).
\end{equation}
Such a bound plays a key role in higher-order estimates. Applying Gr\"onwall's lemma to \eqref{energym-1} and using \eqref{q2}, \eqref{delta1} and \eqref{LinftyEPs}, we obtain
\begin{equation}\label{5.141}
\begin{aligned}
&\quad\sup_{t'\in[0,T]}\|(\rho-1,q)(t')\|_{H^{m-1}}^2
+\frac{1}{\eps}\int_0^T \|q(t')\|_{H^{m-1}}^2\,dt'
\leq C_6e^{C_6E_0^2}E_0^2.
\end{aligned}
\end{equation}

Moreover, together with \eqref{q2}, \eqref{delta1}, \eqref{LinftyEPs} and \eqref{5.141}, it follows from \eqref{rhoL21}--\eqref{nablarhok} in Lemma \ref{lemma4.4} that
\begin{equation}\label{bound:L2n}
\begin{aligned}
&\quad\eps\int_0^T \big(\|(\rho-1)(t')\|_{H^{m}}^2+\|\nabla\phi(t')\|_{L^{2}}^2\big)\,dt'\\
&\leq C_7e^{C_7\varepsilon \int_0^T \|\nabla\phi(t')\|_{L^\infty}^2\,dt'}
\bigg(\|\rho_0-1\|_{H^{m-1}}^2+\|\dive q_0\|_{H^{m-1}}^2\\
&\quad+\varepsilon^2 \Big(\sup_{t'\in[0,T]}\| q(t')\|_{H^{m}}^2
+\frac{1}{\varepsilon} \int_0^T \| q(t')\|_{H^{m}}^2\,dt'\Big)\\
&\quad+\varepsilon\int_0^T\|q(t')\|_{H^{m}}^2 \|v(t')\|_{H^{m}}^2\,dt'\bigg)\\
&\leq C_8 e^{C_8E_0^2}(1+C_*^2\eps^2)E_0^2,
\end{aligned}
\end{equation}
which will be used to establish the highest-order energy estimates of the solution. Note that the above dissipation of $\rho$ does not appear in the definition of $E_T$, so we do not need to include the quantity in the bootstrap assumptions.

 \begin{itemize}
     \item {\emph{Step 3:  Highest-order energy estimates.}}
 \end{itemize}

For the highest-order energy estimate from Lemma \ref{lemma45}, we impose
\begin{align}\label{delta2}
C_*E_0\eps\leq \delta_0^*,
\end{align}
with some sufficiently small uniform constant $\delta_0^*$ such that 
\begin{align}\label{5.14}
1-C_9\eps \|\nabla v(t)\|_{H^{m-1}}\geq1- C_9 C_* E_0 \eps \geq \frac{1}{2}.
\end{align}
Therefore, Lemma \ref{lemma45} implies
\begin{equation}\label{Hkkk}
\begin{aligned}
&\quad \sup_{t\in[0,T]}\|(n,v)(t)\|_{\dot{H}^m}^2
+\frac{1}{\eps}\int_0^T \|v(t')\|_{\dot{H}^m}^2\,dt'\\
&\leq C_{10}\Big(\|(n_0,v_0)\|_{\dot{H}^m}^2
+\frac{1}{\eps}\int_0^T \|v(t')\|_{H^{m-1}}^2\,dt'
+\varepsilon\int_0^T \|n(t')\|_{H^m}^4\,dt' \Big).
\end{aligned}
\end{equation}
Since
\[
\rho_1\le \rho_0\le \rho_2,\qquad
v_0=\left(\frac{1}{\rho_0}-1\right)q_0+q_0,
\]
by Moser-type inequalities \eqref{F1.1311}, we have
\begin{equation}\label{3.1}
\|n_0\|_{H^m}+\|v_0\|_{H^m}\le C_{11}E_0(1+E_0).
\end{equation}
Similarly, due to \eqref{aprioriupper} and the estimates of composite functions (see \eqref{F1.15}), we have $\|n\|_{\dot H^k}\sim \|\rho-1\|_{\dot H^k}$ and
\begin{align}
&\quad \sup_{t'\in[0,T]}\|q(t')\|_{H^m}^2
+\frac{1}{\eps}\int_0^T \|q(t')\|_{H^m}^2\,dt'\nonumber\\
&\leq C_{12}
\Big(1+\sup_{t'\in[0,T]}\|\rho(t')-1\|_{H^{m-1}}^2\Big)
\Big(\sup_{t'\in[0,T]}\|v(t')\|_{H^m}^2
+\frac{1}{\eps}\int_0^T \|v(t')\|_{H^m}^2\,dt'\Big).\label{3.11}
\end{align}
Furthermore,
\begin{align}\label{3.1111}
\frac{1}{\eps}\int_0^T \|v(t')\|_{H^{m-1}}^2\,dt'
\leq C_{13}
\Big(1+\sup_{t'\in[0,T]}\|\rho(t')-1\|_{H^{m-1}}^2\Big)
\frac{1}{\eps}\int_0^T \|q(t')\|_{H^{m-1}}^2\,dt'.
\end{align}
Combining \eqref{5.141}, \eqref{Hkkk} and \eqref{3.1}--\eqref{3.1111}, we arrive at
\begin{align*}
&\quad \sup_{t'\in[0,T]}\big(\|\rho(t')-1\|_{\dot{H}^m}^2+\|q(t')\|_{\dot{H}^m}^2\big)
+\frac{1}{\eps}\int_0^T \|q(t')\|_{\dot{H}^m}^2\,dt'\\
&\leq C_{14}\Big(1+e^{C_{14}(1+E_0^2)}E_0^2\Big)
\Big(E_0^2(1+E_0)^2
+\frac{1}{\eps}\int_0^T\|q(t')\|_{H^{m-1}}^2\,dt'\\
&\quad
+  \varepsilon \int_0^T \|\rho(t')-1\|_{H^{m}}^4\,dt'\Big).
\end{align*}
Note that the $L^2(0,T;H^{m})$-bound of $\rho-1$ has been obtained in \eqref{bound:L2n}.  Using Gr\"onwall's inequality, \eqref{5.141}, \eqref{bound:L2n}, and the elementary fact that
$s\le 1+e^s\le e^{e^s}$ for $s\ge0$, we obtain
\begin{align*}
&\quad \sup_{t'\in[0,T]}\big(\|\rho(t')-1\|_{\dot{H}^m}^2+\|q(t')\|_{\dot{H}^m}^2\big)
+\frac{1}{\eps}\int_0^T \|q(t')\|_{\dot{H}^m}^2\,dt'\\
&\leq C_{15}e^{C_{15}\big(1+e^{C_{15}(1+E_0^2)}E_0^2\big)
\eps\int_0^T\|\rho(t')-1\|_{H^{m}}^2\,dt'}
\Big(1+e^{C_{15}(1+E_0^2)}E_0^2\Big)
\Big(E_0^2(1+E_0)^2
+\frac{1}{\eps}\int_0^T\|q(t')\|_{H^{m-1}}^2\,dt'\Big)\\
&\leq C_{16} e^{C_{16}e^{C_{16}E_0^2}(E_0^2+\eps^2 C_*^2E_0^2)} E_0^2.
\end{align*}
Together with \eqref{delta1} and \eqref{5.141}, this implies the bound of $E_T$:
\begin{align}\label{ET:ess}
E_T\leq C_{17} e^{C_{17}e^{C_{17}E_0^2}E_0^2}E_0.
\end{align}
Now we choose
\begin{align}\label{C*}
C_*:=2C_{17} e^{C_{17}e^{C_{17}E_0^2} E_0^2}.
\end{align}
Then \eqref{ET:ess} gives the desired improved estimate \eqref{apriori:e}.

 \begin{itemize}
     \item{\emph{Step 4:  Upper and lower bound of $\rho$}}
 \end{itemize}

Now, we justify \eqref{lowupper}, which relies strongly on the limiting solution of the drift--diffusion system. Indeed, based on the Sobolev regularity estimates obtained before, one can control the error estimate \eqref{2.21} between  $\rho^\eps(t,x)$ and $\rho^*(\eps t,x)$, which, together with the maximum principle of the limiting solution $\rho^*$,  leads to the upper and lower bounds of $\rho$. 

Since \eqref{5.141}, \eqref{bound:L2n} and  \eqref{ET:ess} have been established, and $\rho^*$ with the same datum $\rho_0$ satisfies \eqref{exp1}, it follows from the definition of $B(t)$ in \eqref{Bt} that
\begin{align}\label{5.23}
B(T)\leq C_{18} e^{C_{18} e^{C_{18} E_0^2}} E_0^2.
\end{align}
Then, the $L^\infty$ estimate \eqref{2.22} for $\rho$ in Lemma \ref{lem2.6} ensures that
\begin{align*}
|\rho(t,x)-\rho^*(\eps t,x)|\leq
C_{19}e^{C_{19}e^{C_{19}E_0^2}}E_0\eps,\quad (t,x)\in [0,T]\times\mathbb{R}^d.
\end{align*}
Together with the maximum principle for $\rho^*$ in \eqref{1.16} and the definitions of $\tilde{\rho}_1$ and $\tilde{\rho}_2$ in   \eqref{tilderho12}, this implies the upper and lower bounds of $\rho$ in \eqref{lowupper} provided that
\begin{align}\label{delta3}
C_{20}e^{C_{20}e^{C_{20}E_0^2}}E_0\eps
\leq \frac12\min\{\rho_1,\rho_2\}.
\end{align}

Since $C_*$ is defined by \eqref{C*}, we can verify that, by choosing $\delta_0>0$ sufficiently small and $\bar{C}$ large, the assumption \eqref{smallness-main} implies the auxiliary conditions \eqref{delta1}, \eqref{delta2} and \eqref{delta3}.

The {\emph{a priori}} bounds \eqref{apriori:e} and \eqref{lowupper} allow us to extend the local solution globally in time. Indeed, suppose by contradiction that $T_\eps<\infty$. Then \eqref{apriori:e} gives a uniform bound of $(\rho,q,\phi)$ in the regularity class required by the local well-posedness theory on $[0,T_\eps)$, while \eqref{lowupper} ensures that the density remains uniformly away from vacuum. Hence, for any $t_0<T_\eps$ sufficiently close to $T_\eps$, the local existence result applied with initial data $(\rho,q,\phi)(t_0)$ yields a continuation of the solution beyond $T_\eps$, which contradicts the maximality of $T_\eps$. Therefore, $T_\eps=\infty$. This completes the global existence of Theorem \ref{theorem1.1}.


\subsection{Exponential stability}\label{section:EP4}

Owing to \eqref{ET:ess}, we obtain the following uniform estimate
\begin{equation}\label{r1}
\begin{aligned}
&\sup_{t'\in \mathbb{R}_+}\big( \|\rho(t')-1\|_{H^m}^2+\|q(t')\|_{H^m}^2+\|\nabla\phi(t')\|_{L^2}^2\big)\\
&\quad +\int_0^\infty\big( \eps\|\rho(t')-1\|_{H^m}^2+\frac{1}{\eps}\|q(t')\|_{H^m}^2+\eps\|\nabla\phi(t')\|_{L^2}^2\big)\,dt'\leq CE_0^2,
\end{aligned}
\end{equation}
where $C>0$ denotes a constant independent of $\eps$.

We define the low-order energy 
\begin{equation}
\begin{aligned}
\mathcal{E}_1(t):&=\int_{\R^d}\Big(E(\rho,v)+\frac12|\nabla\phi|^2\Big)\,dx\\
&\quad+ \sum_{1\leq |\alpha|\leq m-1}\Big(a^2\|\partial^\alpha \rho\|_{L^2}^2+\|\Lambda^{-1}\partial^\alpha\rho\|_{L^2}^2+\|\partial^\alpha q\|_{L^2}^2\Big),
\end{aligned}
\end{equation}
and the highest-order energy
\begin{align*}
\mathcal{E}_2(t):=\sum_{|\alpha|=m}\Big(\|\partial^\alpha n\|_{L^2}^2+\|\partial^\alpha v\|_{L^2}^2+\frac{1}{a^2}\|\Lambda^{-1}\partial^\alpha n\|_{L^2}^2\Big),
\end{align*}
where $n$ is given by \eqref{def:nv}. It thus follows from \eqref{3.13}, \eqref{ddtHk} and \eqref{Em-top} that
\begin{align*}
&\quad\frac{d}{dt}\Big( \mathcal{E}_1(t)+\beta_1 \mathcal{E}_2(t) \Big)+\frac{1}{\eps}\int_{\mathbb{R}^d}\rho |v|^2\,dx+\frac{1}{\eps}\sum_{1\leq |\alpha|\leq m-1}\|\partial^\alpha q\|_{L^2}^2-\frac{C\beta_1}{\eps}\|q\|_{H^{m-1}}^2\\
&\quad\quad+\frac{\beta_1}{\varepsilon}(1-C \eps\|\nabla v\|_{L^\infty})\sum_{|\alpha|=m}\|\partial^\alpha v\|_{L^2}^2\\
&\leq C(1+\beta_1)\eps \|n\|_{H^m}^4.
\end{align*}
To derive the dissipation of $\rho$, we also introduce the corrector
\begin{align*}
\mathcal{E}_3(t):&=\|\Lambda^{-1}(\rho-1)\|_{L^2}^2+\eps \langle q, \Lambda^{-2}\grad (\rho-1)\rangle+\|\rho-1\|_{L^2}^2-\eps \langle \dive q, \rho-1\rangle\\
&+\sum_{1\leq |\alpha|\leq m-1}\Big(\|\partial^\alpha\rho\|_{L^2}^2-\eps \langle \dive \partial^\alpha q, \partial^\alpha \rho\rangle \Big).
\end{align*}
Thus, after similar calculations as in Lemma \ref{lemma4.4}, we can deduce the following inequality:
\begin{align*}
&\quad\frac{d}{dt}\mathcal{E}_3(t)+\eps a^2 \|\Lambda^{-1}\nabla (\rho-1)\|_{L^2}^2+\eps \int_{\mathbb{R}^d} \rho |\Lambda^{-2}\nabla(\rho-1)|^2\, dx-\eps\|\Lambda^{-1}\dive q\|_{L^2}^2\\
&\quad\quad+\eps a^2 \|\nabla \rho\|_{L^2}^2+\eps \|\rho-1\|_{L^2}^2+\eps \int_{\mathbb{R}^d} \rho(\rho-1)^2\,dx-\eps\|\dive q\|_{L^2}^2\\
&\quad\quad+\eps a^2 \sum_{1\leq |\alpha|\leq m-1}\|\nabla \partial^\alpha \rho\|_{L^2}^2+\eps \sum_{1\leq |\alpha|\leq m-1} (\|\partial^\alpha \rho\|_{L^2}^2-\|\dive \partial^\alpha  q\|_{L^2}^2)\\
&\leq C\eps \|q\|_{H^m}^2 \|v\|_{H^m}^2+C\eps\|\nabla\Lambda^{-2} (\rho-1)\|_{L^\infty}^2 \|\rho-1\|_{H^m}^2.
\end{align*}

Now we define the functional
\begin{align*}
\mathcal{E}(t):=\mathcal{E}_1(t)+\beta_1\mathcal{E}_2(t)+\beta_2\mathcal{E}_3(t).
\end{align*}
Employing \eqref{lowupper0}, \eqref{r1} and \eqref{5.14} and choosing two suitably small positive constants $\beta_1, \beta_2$, we can find uniform-in-time constants $\lambda^*>0$ and $C^*>1$ such that
\begin{equation}\label{sim}
\begin{aligned}
 \frac{1}{C^*} \mathcal{E}(t) \leq \|\rho(t)-1\|_{H^m}^2+\|q(t)\|_{H^m}^2+\|\nabla\phi(t)\|_{L^2}^2\leq C^* \mathcal{E}(t)
\end{aligned}
\end{equation}
and
\begin{equation}\label{LY}
\begin{aligned}
\frac{d}{dt}\mathcal{E}(t)+2\lambda^* \eps \mathcal{E}(t)+\frac{\lambda^*}{\eps}\|q\|_{H^m}^2\leq C\eps \big(\|\nabla\Lambda^{-2} (\rho-1)\|_{L^\infty}^2+\|q\|_{H^m}^2\big)\mathcal{E}(t).
\end{aligned}
\end{equation}
This yields
\begin{equation}
\begin{aligned}
&\quad\frac{d}{dt}\Big(e^{\lambda^*\eps t}\mathcal{E}(t)\Big)+\lambda^* \eps  e^{\lambda^*\eps t}\mathcal{E}(t)+\frac{\lambda^*}{\eps}e^{\lambda^*\eps t}\|q\|_{H^m}^2\\
&\leq C\eps \big(\|\nabla\Lambda^{-2} (\rho-1)\|_{L^\infty}^2+\|q\|_{H^m}^2\big)e^{\lambda^*\eps t} \mathcal{E}(t).
\end{aligned}
\end{equation}
Since $\|\nabla\Lambda^{-2} (\rho-1)\|_{L^\infty}^2+\|q\|_{H^m}^2\in L^1(\mathbb{R}_+)$, Gr\"onwall's inequality gives rise to
\begin{equation}
\begin{aligned}\label{Etexp}
&\quad e^{\lambda^*\eps t}\mathcal{E}(t)+\int_0^t e^{\lambda^*\eps t'}\big(\eps \mathcal{E}(t')+\frac{1}{\eps}\|q(t')\|_{H^m}^2\big)\,dt'\\
&\leq e^{C\eps \int_0^\infty\big(\|\nabla\Lambda^{-2}(\rho-1)(t')\|_{L^\infty}^2+\|q(t')\|_{H^m}^2\big)\,dt'}\mathcal{E}(0).
\end{aligned}
\end{equation}
In view of \eqref{sim} and \eqref{Etexp}, we arrive at \eqref{exp:EP} and finish the proof of Theorem \ref{theorem1.1}.

\section{Proof of Theorem \ref{theorem1.3}}\label{section:error}

Combining \eqref{exp:EP1} with the unweighted estimates obtained in Lemmas \ref{lemma4555} and \ref{lem2.5}, we first have
\begin{equation}
\begin{aligned}\label{6.9}
&\quad \sup_{t'\in \mathbb{R}_+}\|(\rho^\eps-\rho^*)(t')\|_{H^{m-1}\cap\dot{H}^{-1}}^2\\
&+\int_0^\infty \big(\|(\rho^\eps-\rho^*)(t')\|_{H^m\cap\dot{H}^{-1}}^2+\|Z^\eps(t')\|_{H^{m-1}}^2\big)\, dt'
\le C\varepsilon ^2 E_0^2.
\end{aligned}
\end{equation}
Now we notice that
\begin{align*}
q^\eps-q^*-q_L^\eps=Z^\eps-a^2\nabla(\rho^\eps-\rho^*)-\big(\rho^\eps \nabla \Lambda^{-2}(\rho^\eps-1)-\rho^*\nabla\Lambda^{-2}(\rho^*-1)\big)+q_{L,1}^\eps,
\end{align*}
which implies 
\begin{equation*}
\begin{aligned}
&\quad\int_0^\infty \|(q^\eps-q^*-q_L^\eps)(t')\|_{H^{m-1}}^2\,dt'\\
&\le C\int_0^\infty \big(\|Z^\eps(t')\|_{H^{m-1}}^2
+\|\nabla(\rho^\eps-\rho^*)(t')\|_{H^{m-1}}^2\\
&\quad+\|(\rho^\eps\nabla\Lambda^{-2}(\rho^\eps-1)-\rho^*\nabla\Lambda^{-2}(\rho^*-1))(t')\|_{H^{m-1}}^2
+\|q_{L,1}^\eps(t')\|_{H^{m-1}}^2\big)\,dt'.
\end{aligned}
\end{equation*}
Here, the Moser product law, together with \eqref{6.9}, indicates that
\begin{equation*}
\begin{aligned}
&\quad\int_{0}^{\infty} \|(\rho^\eps \nabla \Lambda^{-2}(\rho^\eps-1)-\rho^*\nabla\Lambda^{-2}(\rho^*-1))(t')\|_{H^{m-1}}^2 dt'\\
&\leq 2\int_{0}^{\infty} \big(\|\big((\rho^\eps-\rho^*)\nabla \Lambda^{-2}(\rho^\eps-1)\big)(t')\|_{H^{m-1}}^2+\|\rho^*\nabla \Lambda^{-2}(\rho^\eps-\rho^*)(t')\|_{H^{m-1}}^2 \big)\,dt'\\
&\leq C\sup_{t'\in \mathbb{R}_+}\|(\rho^\eps-\rho^*)(t')\|_{H^{m-1}}^2 \int_{0}^{\infty}\|\nabla\phi^\eps(t')\|_{H^{m-1}}^2\,dt'\\
&\quad+C(1+\sup_{t'\in  \mathbb{R}_+}\|\rho^*(t')-1\|_{H^{m-1}}^2)\int_{0}^{\infty}\|\nabla (\phi^\eps-\phi^*)(t')\|_{H^{m-1}}^2\,dt'\leq CE_0^2 \eps^2.
\end{aligned}
\end{equation*}
Therefore, together with \eqref{tlldeqL2} and \eqref{6.9}, it holds
\begin{align}\label{step:113ff}
    \int_0^\infty \|(q^\eps-q^*-q_L^\eps)(t')\|_{H^{m-1}}^2\,dt'\leq CE_0^2 \eps^2.
\end{align}

We now establish the improved exponential stability estimates. Let $\lambda_1>0$ be a small uniform constant such that
\[
\lambda_1\leq \min\left\{\lambda_0,\lambda_0^*,\frac12\right\}.
\]
In particular, since $0<\eps\le1$, we have $\lambda_1\le 1/(2\eps^2)$.
Since $Z^\eps(0)=0$, multiplying the differential form of \eqref{Z-energy-raw} by
$e^{\lambda_1t}$ and integrating over $[0,t]$, we obtain
\begin{equation}
\begin{aligned}\label{6.12}
&\quad \frac12 e^{\lambda_1 t}\|Z^\eps(t)\|_{H^{m-1}}^2
+\left(\frac{1}{\eps^2}-\frac{\lambda_1}{2}\right)
\int_0^t e^{\lambda_1 t'}\|Z^\eps(t')\|_{H^{m-1}}^2\,dt'\\
&=\int_0^t e^{\lambda_1 t'}\big\langle a^2\nabla\partial_t\rho^\eps,Z^\eps\big\rangle_{H^{m-1}}\,dt'
+\int_0^t e^{\lambda_1 t'}\big\langle \tilde f^\eps,Z^\eps\big\rangle_{H^{m-1}}\,dt'.
\end{aligned}
\end{equation}
We first handle the coupling term. Recalling the definition of $Z^\eps$,  we write
\begin{align*}
&\quad \int_0^t e^{\lambda_1 t'}
\big\langle a^2\nabla\partial_t\rho^\eps,Z^\eps\big\rangle_{H^{m-1}}\,dt'\\
&=\int_0^t e^{\lambda_1 t'}
\big\langle a^2\nabla\partial_t\rho^\eps,a^2\nabla\rho^\eps\big\rangle_{H^{m-1}}\,dt'\\
&\quad-\int_0^t e^{\lambda_1 t'}
\big\langle a^2\nabla\partial_t\rho^\eps,q_{L,1}^\eps\big\rangle_{H^{m-1}}\,dt'+\int_0^t e^{\lambda_1 t'}
\big\langle a^2\nabla\partial_t\rho^\eps,
q^\eps-q_L^\eps+\rho^\eps\nabla\Lambda^{-2}(\rho^\eps-1)
\big\rangle_{H^{m-1}}\,dt'.
\end{align*}
For the first term, we have
\begin{align*}
&\quad \int_0^t e^{\lambda_1 t'}
\big\langle a^2\nabla\partial_t\rho^\eps,a^2\nabla\rho^\eps\big\rangle_{H^{m-1}}\,dt'\\
&=\frac{a^4}{2}e^{\lambda_1 t}\|\nabla\rho^\eps(t)\|_{H^{m-1}}^2
-\frac{a^4}{2}\|\nabla\rho_0\|_{H^{m-1}}^2
-\frac{a^4\lambda_1}{2}\int_0^t e^{\lambda_1 t'}
\|\nabla\rho^\eps(t')\|_{H^{m-1}}^2\,dt'\\
&\le C e^{\lambda_1 t}\|\rho^\eps(t)-1\|_{H^m}^2.
\end{align*}
For the term involving $q_{L,1}^\eps$, integrating by parts in time gives
\begin{align*}
&\quad \left|\int_0^t e^{\lambda_1 t'}
\big\langle a^2\nabla\partial_t\rho^\eps,q_{L,1}^\eps\big\rangle_{H^{m-1}}\,dt'\right|\\
&\le
\left|e^{\lambda_1 t}\big\langle a^2\nabla\rho^\eps(t),q_{L,1}^\eps(t)\big\rangle_{H^{m-1}}
-\big\langle a^2\nabla\rho_0,q_{L,1}^\eps(0)\big\rangle_{H^{m-1}}\right|\\
&\quad+\int_0^t e^{\lambda_1 t'}
\left|\big\langle a^2\nabla\rho^\eps,
\partial_t q_{L,1}^\eps\big\rangle_{H^{m-1}}\right|\,dt'
+\lambda_1\int_0^t e^{\lambda_1 t'}
\left|\big\langle a^2\nabla\rho^\eps,q_{L,1}^\eps\big\rangle_{H^{m-1}}\right|\,dt'\\
&\le C\sup_{0\le t'\le t}e^{\lambda_1 t'}\|\rho^\eps(t')-1\|_{H^m}^2+C(1+\|\rho_0-1\|_{H^{m-1}}^2)
\bigl(\|\nabla\phi_0\|_{L^2}^2+\|\rho_0-1\|_{H^m}^2\bigr).
\end{align*}
For the remaining part of the coupling term, using
$\partial_t\rho^\eps=-\dive q^\eps$ and integrating by parts, we obtain
\begin{align*}
&\quad\left|
\big\langle a^2\nabla\partial_t\rho^\eps,
q^\eps-q_L^\eps+\rho^\eps\nabla\Lambda^{-2}(\rho^\eps-1)
\big\rangle_{H^{m-1}}
\right|\\
&\le C\|q^\eps\|_{H^m}^2
+C\|q_L^\eps\|_{H^m}^2
+C\|\dive(\rho^\eps\nabla\Lambda^{-2}(\rho^\eps-1))\|_{H^{m-1}}^2\\
&\le C\|q^\eps\|_{H^m}^2
+C\|q_L^\eps\|_{H^m}^2
+C(1+\|\rho^\eps-1\|_{H^m}^2)\|\rho^\eps-1\|_{H^m}^2.
\end{align*}
Therefore,
\begin{align}
&\quad \int_0^t e^{\lambda_1 t'}
\big\langle a^2\nabla\partial_t\rho^\eps,Z^\eps\big\rangle_{H^{m-1}}\,dt' \nonumber\\
&\le C\sup_{0\le t'\le t}e^{\lambda_1 t'}\|\rho^\eps(t')-1\|_{H^m}^2
+C\int_0^t e^{\lambda_1 t'}\|q^\eps(t')\|_{H^m}^2\,dt' \nonumber\\
&\quad+C\int_0^t e^{\lambda_1 t'}\|q_L^\eps(t')\|_{H^m}^2\,dt'
+C\int_0^t e^{\lambda_1 t'}
(1+\|\rho^\eps(t')-1\|_{H^m}^2)
\|\rho^\eps(t')-1\|_{H^m}^2\,dt' \nonumber\\
&\quad+C(1+\|\rho_0-1\|_{H^{m-1}}^2)
\bigl(\|\nabla\phi_0\|_{L^2}^2+\|\rho_0-1\|_{H^m}^2\bigr).
\label{coupling-exp}
\end{align}
Next, using Young's inequality, we get
\begin{align}
&\quad \left|\int_0^t e^{\lambda_1 t'}
\big\langle \tilde f^\eps,Z^\eps\big\rangle_{H^{m-1}}\,dt'\right| \nonumber\\
&\le \frac{1}{4\eps^2}
\int_0^t e^{\lambda_1 t'}\|Z^\eps(t')\|_{H^{m-1}}^2\,dt'
+C\eps^2\int_0^t e^{\lambda_1 t'}
\|\tilde f^\eps(t')\|_{H^{m-1}}^2\,dt'.
\label{fZ-exp}
\end{align}
By \eqref{f-tilde-final}, we have
\begin{align}
\eps^2\|\tilde f^\eps\|_{H^{m-1}}^2
\le
C\eps^2(1+\|\rho^\eps-1\|_{H^m}^2)\|q^\eps\|_{H^m}^2
+C\eps^2\|q^\eps\|_{H^m}^2\|v^\eps\|_{H^m}^2.
\label{f-tilde-exp}
\end{align}
Combining \eqref{6.12}--\eqref{f-tilde-exp}, using
$\lambda_1\le 1/(2\eps^2)$, and taking the supremum in time, we infer
\begin{equation}\label{2.1700}
\begin{aligned}
&\quad \sup_{t'\in\mathbb{R}_+}e^{\lambda_1 t'}\|Z^\eps(t')\|_{H^{m-1}}^2
+\frac{1}{\eps^2}\int_0^{\infty} e^{\lambda_1 t'} \|Z^\eps(t')\|_{H^{m-1}}^2\,dt'\\
&\leq C\sup_{t'\in\mathbb{R}_+}e^{\lambda_1 t'}\|\rho^\eps(t')-1\|_{H^m}^2+C\int_0^\infty e^{\lambda_1t'}\|q_L^\eps(t')\|_{H^m}^2\,dt'\\
&\quad+C\int_0^{\infty}e^{\lambda_1 t'}
\bigl(\|q^\eps(t')\|_{H^m}^2+\|\rho^\eps(t')-1\|_{H^m}^2\bigr)\,dt'\\
&\quad+C(1+\|\rho_0-1\|_{H^{m-1}}^2)
\bigl(\|\nabla\phi_0\|_{L^2}^2+\|\rho_0-1\|_{H^m}^2\bigr).
\end{aligned}
\end{equation}
Here we have used the uniform bounds in \eqref{exp:EP1} to absorb the quadratic
coefficients generated by \eqref{f-tilde-exp}. Moreover,
\begin{align*}
&\int_0^\infty e^{\lambda_1 t'} \|q_{L}^\eps(t')\|_{H^m}\,dt'
\leq
\int_0^\infty e^{-(\frac{1}{\eps^2}-\lambda_1)t'}\,dt'\,
\frac{1}{\eps}\|q_0\|_{H^m}
\leq C\eps \|q_0\|_{H^m},\\
&\int_0^\infty e^{\lambda_1 t'} \|q_{L}^\eps(t')\|_{H^m}^2\,dt'
\leq
\int_0^\infty e^{-(\frac{2}{\eps^2}-\lambda_1)t'}\,dt'\,
\frac{1}{\eps^2}\|q_0\|_{H^m}^2
\leq C \|q_0\|_{H^m}^2,\\
&\int_0^\infty e^{\lambda_1 t'} \|q_{L,1}^\eps(t')\|_{H^{m-1}}^2\,dt'
\leq
C\eps^2
(1+\|\rho_0-1\|_{H^{m-1}}^2)
\bigl(\|\nabla\phi_0\|_{L^2}^2+\|\rho_0-1\|_{H^m}^2\bigr).
\end{align*}
Thus, by \eqref{exp:EP1} and \eqref{2.1700}, we obtain
\begin{equation*}
\begin{aligned}
&\quad \sup_{t'\in\mathbb{R}_+}e^{\lambda_1 t'}\|Z^\eps(t')\|_{H^{m-1}}^2
+\frac{1}{\eps^2}\int_0^{\infty} e^{\lambda_1 t'} \|Z^\eps(t')\|_{H^{m-1}}^2\,dt'
\le C E_0^2.
\end{aligned}
\end{equation*}
Consequently, one can carry out similar calculations as in Lemma \ref{lemma4555} to infer
\begin{equation*}
\begin{aligned}
&\quad \sup_{t'\in \mathbb{R}_+} e^{\lambda_1 t'}
\|(\rho^\eps-\rho^*)(t')\|_{H^{m-1}\cap\dot{H}^{-1}}^2
+\int_0^{\infty}e^{\lambda_1 t'}
\|(\rho^\eps-\rho^*)(t')\|_{H^{m}\cap\dot{H}^{-1}}^2\,dt'\\
&\leq C
\Big(
\frac{1}{\eps^2}\int_0^{\infty} e^{\lambda_1 t'}
\|Z^\eps(t')\|_{H^{m-1}}^2\,dt'+\|q_0\|_{H^m}^2\\
&\quad+(1+\|\rho_0-1\|_{H^{m-1}}^2)(\|\nabla\phi_0\|_{L^2}^2+\|\rho_0-1\|_{H^{m}}^2)
\Big)\eps^2
\leq C E_0^2\eps^2.
\end{aligned}
\end{equation*}
Since $\|\rho^\eps-\rho^*\|_{\dot{H}^{-1}}
\sim \|\nabla(\phi^\eps-\phi^*)\|_{L^2}$, one has
\begin{equation*}
\begin{aligned}
&\quad \sup_{t'\in \mathbb{R}_+} e^{\lambda_1 t'}
\|\nabla(\phi^\eps-\phi^*)(t')\|_{L^2}^2
+\int_0^{\infty}e^{\lambda_1 t'}
\|\nabla(\phi^\eps-\phi^*)(t')\|_{L^2}^2\,dt'
\leq C E_0^2\eps^2.
\end{aligned}
\end{equation*}
Repeating the above argument with the weight $e^{\lambda_1 t}$, and using the weighted bounds for $Z^\eps$, $\rho^\eps-\rho^*$, $q_{L,1}^\eps$ and the exponential estimates for $\rho^\eps,\rho^*$, we obtain
\[
\int_0^\infty e^{\lambda_1 t}
\|(q^\eps-q^*-q_L^\eps)(t)\|_{H^{m-1}}^2\,dt
\le C E_0^2\eps^2.
\]
Consequently, the bound \eqref{error2} follows, which leads to the strong convergence properties \eqref{strong1} and \eqref{strong10} directly. In addition, the convergence of the momentum and Darcy's law in
$L^2(1,\infty;H^{m-1})$ can be obtained modulo the initial layer corrections $q_L^\eps$ and $Z_L^\eps$. Note that there exists a constant $C'$ independent of $\eps$ such that, for $\eps$ sufficiently small,
\[
\begin{aligned}
\int_{1}^{\infty} e^{\lambda_1 t'}\|q_L^\eps(t')\|_{H^{m}}^2\,dt'
&=\frac{1}{\eps^2}\int_{1}^{\infty}
e^{-\left(\frac{2}{\eps^2}-\lambda_1\right)t'}\,dt'\,
\|q_0\|_{H^m}^2=\frac{1}{\eps^2}
\frac{e^{-\left(\frac{2}{\eps^2}-\lambda_1\right)}}{\frac{2}{\eps^2}-\lambda_1}
\|q_0\|_{H^m}^2\leq C'\eps^2\|q_0\|_{H^m}^2 
\end{aligned}
\]
and
\[
\int_{1}^{\infty} e^{\lambda_1 t'}\|Z_L^\eps(t')\|_{H^{m-1}}^2\,dt'
\leq C'\eps^2\bigl(\|q_0\|_{H^{m-1}}^2+\|\nabla p(\rho_0)\|_{H^{m-1}}^2
+\|\rho_0\nabla\phi_0\|_{H^{m-1}}^2\bigr).
\]
Combining these with \eqref{error2}, we justify \eqref{strong2}-\eqref{strong3}, and thus the proof of Theorem \ref{theorem1.3} is complete.

\bigbreak

\bigbreak

\noindent
{\textbf{Acknowledgments:}}  L.-Y. Shou is supported by the National Natural Science Foundation of China (12301275). J. Xu is partially supported by the National Natural Science Foundation of China (12271250). 

\bigbreak
\noindent
{\textbf{Data availability statement:}} Data sharing is not applicable to this article, as no datasets were generated or analyzed during the current study.



\vspace{4mm}

\vspace{4mm}

(Y.-J. Peng)\par\nopagebreak

\textsc{School of Mathematics, Nanjing University of Aeronautics and
Astronautics, Nanjing, 211106, China}

\textsc{Laboratoire de Mathématiques Blaise Pascal,
	Université Clermont Auvergne / CNRS
	63178 Aubière Cedex, France}

Email address: {\texttt{yue-jun.peng@uca.fr}}

\vspace{3ex}

(L.-Y. Shou)\par\nopagebreak

\textsc{School of Mathematical Sciences, Ministry of Education Key Laboratory of NSLSCS, and Key Laboratory of Jiangsu Provincial Universities of FDMTA, Nanjing Normal University, Nanjing 210023, China}

Email address: {\texttt{shoulingyun11@gmail.com}}

\vspace{3ex}

(J. Xu)\par\nopagebreak

\textsc{School of Mathematics, Nanjing University of Aeronautics and
Astronautics, Nanjing, 211106, China}

Email address: {\texttt{jiangxu\_79@nuaa.edu.cn}}


\vspace{5mm}
\begin{thebibliography}{99}
\vspace{3mm}

\bibitem{A1}
G. Al\`i, Global existence of smooth solutions of the $N$-dimensional Euler--Poisson model, \emph{SIAM J. Math. Anal.} \textbf{35} (2003), 389--422.

\bibitem{ABR1}
G. Al\`i, D. Bini, and S. Rionero, Global existence and relaxation limit for smooth solutions to the Euler--Poisson model for semiconductors, \emph{SIAM J. Math. Anal.} \textbf{32} (2000), 572--587.

\bibitem{c1}
F. Chen, \emph{Introduction to Plasma Physics and Controlled Fusion}, Vol. 1, Plenum Press, New York, 1984.

\bibitem{CHWY}
G.-Q. G. Chen, L. He, Y. Wang, and D. Yuan, Global solutions of the compressible Euler--Poisson equations with large initial data of spherical symmetry, \emph{Comm. Pure Appl. Math.} \textbf{77} (2024), no. 6, 2947--3025.


\bibitem{1}
J.-F. Coulombel and T. Goudon, The strong relaxation limit of the multidimensional isothermal Euler equations, \emph{Trans. Amer. Math. Soc.} \textbf{359} (2007), 637--648.

\bibitem{c3}
T. Crin-Barat and R. Danchin, Global existence for partially dissipative hyperbolic systems in the $L^p$ framework, and relaxation limit, \emph{Math. Ann.} \textbf{386} (2023), 2159--2206.

\bibitem{CBPS}
T. Crin-Barat, Y.-J. Peng, and L.-Y. Shou, Global convergence rates in the relaxation limits for the compressible Euler and Euler--Maxwell systems in Sobolev spaces, \emph{J. Differential Equations} \textbf{453} (2026), Paper No. 113805, 63 pp.

\bibitem{CBPSX}
T. Crin-Barat, Y.-J. Peng, L.-Y. Shou, and J. Xu, A new characterization of the dissipation structure and the relaxation limit for the compressible Euler--Maxwell system, \emph{J. Funct. Anal.} \textbf{289} (2025), Paper No. 110918, 51 pp.

\bibitem{Dafermos1}
C. M. Dafermos, \emph{Hyperbolic Conservation Laws in Continuum Physics}, 3rd ed., Grundlehren Math. Wiss., vol. 325, Springer, Heidelberg, 2010.


\bibitem{GMP1}
P. Germain, N. Masmoudi, and B. Pausader, Nonneutral global solutions for the electron Euler--Poisson system in three dimensions, \emph{SIAM J. Math. Anal.} \textbf{45} (2013), 267--278.


\bibitem{Guo1}
Y. Guo, Smooth irrotational flows in the large to the Euler--Poisson system in $\mathbb{R}^{3+1}$, \emph{Comm. Math. Phys.} \textbf{195} (1998), 249--265.

\bibitem{Guo2}
Y. Guo and B. Pausader, Global smooth ion dynamics in the Euler--Poisson system, \emph{Comm. Math. Phys.} \textbf{303} (2011), 89--125.

\bibitem{guoblow}
Y. Guo and A. S. Tahvildar-Zadeh, Formation of singularities in relativistic fluid dynamics and in spherically symmetric plasma dynamics, in: \emph{Nonlinear Partial Differential Equations} (Evanston, IL, 1998), Contemp. Math., vol. 238, Amer. Math. Soc., Providence, RI, 1999, pp. 151--161.

\bibitem{HP}
M.-L. Hajjej and Y.-J. Peng, Initial layers and zero-relaxation limits of multidimensional Euler--Poisson equations, \emph{Math. Methods Appl. Sci.} \textbf{36} (2013), 182--195.

\bibitem{HMW}
L. Hsiao, P. A. Markowich, and S. Wang, The asymptotic behavior of globally smooth solutions of the multidimensional isentropic hydrodynamic model for semiconductors, \emph{J. Differential Equations} \textbf{192} (2003), 111--133.

\bibitem{IP1}
A. D. Ionescu and B. Pausader, The Euler--Poisson system in 2D: global stability of the constant equilibrium solution, \emph{Int. Math. Res. Not. IMRN} (2013), no. 4, 761--826.

\bibitem{J}
J. D. Jackson, \emph{Classical Electrodynamics}, 2nd ed., John Wiley, New York, 1975.

\bibitem{5}
S. Junca and M. Rascle, Strong relaxation of the isothermal Euler system to the heat equation, \emph{Z. Angew. Math. Phys.} \textbf{53} (2002), 239--264.

\bibitem{JP1}
A. J\"ungel and Y.-J. Peng, A hierarchy of hydrodynamic models for plasmas: zero-relaxation-time limits, \emph{Comm. Partial Differential Equations} \textbf{24} (1999), 1007--1033.

\bibitem{JP2}
A. J\"ungel and Y.-J. Peng, Zero-relaxation-time limits in hydrodynamic models for plasmas revisited, \emph{Z. Angew. Math. Phys.} \textbf{51} (2000), 385--396.

\bibitem{KatoPonce1988}
T. Kato and G. Ponce, Commutator estimates and the Euler and Navier--Stokes equations, \emph{Comm. Pure Appl. Math.} \textbf{41} (1988), no. 7, 891--907.

\bibitem{Lady}
O. A. Ladyzhenskaya, V. A. Solonnikov, and N. N. Ural'tseva, \emph{Linear and Quasilinear Equations of Parabolic Type}, Transl. Math. Monogr., vol. 23, Amer. Math. Soc., Providence, RI, 1968.


\bibitem{LM1}
C. Lattanzio and P. Marcati, The relaxation to the drift--diffusion system for the 3-D isentropic Euler--Poisson model for semiconductors, \emph{Discrete Contin. Dyn. Syst.} \textbf{5} (1999), 449--455.

\bibitem{LT1}
C. Lattanzio and A. E. Tzavaras, From gas dynamics with large friction to gradient flows describing diffusion theories, \emph{Comm. Partial Differential Equations} \textbf{42} (2017), no. 2, 261--290.

\bibitem{7}
P. D. Lax, Hyperbolic systems of conservation laws II, \emph{Comm. Pure Appl. Math.} \textbf{10} (1957), 537--566.


\bibitem{LW1}
D. Li and Y. Wu, The Cauchy problem for the two-dimensional Euler--Poisson system, \emph{J. Eur. Math. Soc.} \textbf{16} (2014), 2211--2266.

\bibitem{LMM}
H.-L. Li, P. A. Markowich, and M. Mei, Asymptotic behaviour of solutions of the hydrodynamic model of semiconductors, \emph{Proc. Roy. Soc. Edinburgh Sect. A} \textbf{132} (2002), no. 2, 359--378.

\bibitem{9}
Y. Li, Y.-J. Peng, and L. Zhao, Convergence rate from hyperbolic systems of balance laws to parabolic systems, \emph{Appl. Anal.} \textbf{100} (2021), no. 5, 1079--1095.

\bibitem{LPZ1}
Y. Li, Y.-J. Peng, and L. Zhao, Convergence rates in zero-relaxation limits for Euler--Maxwell and Euler--Poisson systems, \emph{J. Math. Pures Appl. (9)} \textbf{154} (2021), 185--211.

\bibitem{10}
C. Lin and J.-F. Coulombel, The strong relaxation limit of the multidimensional Euler equations, \emph{NoDEA Nonlinear Differential Equations Appl.} \textbf{20} (2013), 447--461.

\bibitem{liu1}
T.-P. Liu, Hyperbolic conservation laws with relaxation, \emph{Comm. Math. Phys.} \textbf{108} (1987), no. 1, 153--175.


\bibitem{11}
A. Majda, \emph{Compressible Fluid Flow and Systems of Conservation Laws in Several Space Variables}, Springer, New York, 1984.

\bibitem{MP}
T. Makino and B. Perthame, Sur les solutions \`a sym\'etrie sph\'erique de l'\'equation d'Euler--Poisson pour l'\'evolution d'\'etoiles gazeuses, \emph{Japan J. Appl. Math.} \textbf{7} (1990), 165--170.

\bibitem{12}
P. Marcati and A. Milani, The one-dimensional Darcy's law as the limit of a compressible Euler flow, \emph{J. Differential Equations} \textbf{84} (1990), 129--147.

\bibitem{MN}
P. Marcati and R. Natalini, Weak solutions to a hydrodynamic model for semiconductors and relaxation to the drift--diffusion equations, \emph{Arch. Ration. Mech. Anal.} \textbf{129} (1995), 129--145.

\bibitem{M1}
P. Marcati and B. Rubino, Hyperbolic to parabolic relaxation theory for quasilinear first order systems, \emph{J. Differential Equations} \textbf{162} (2000), 359--399.

\bibitem{MR1}
P. A. Markowich, C. A. Ringhofer, and C. Schmeiser, \emph{Semiconductor Equations}, Springer-Verlag, New York, 1990.

\bibitem{N1}
R. Natalini, Recent results on hyperbolic relaxation problems, in: \emph{Analysis of Systems of Conservation Laws} (Aachen, 1997), Monogr. Surv. Pure Appl. Math., vol. 99, Chapman \& Hall/CRC, Boca Raton, FL, 1999, pp. 128--198.

\bibitem{Ni68} T. Nishida,
Global solution for an initial boundary value problem of a quasilinear
hyperbolic system,
\emph{Proc. Japan Acad.} \textbf{44} (1968), 642--646.

\bibitem{NS} T. Nishida, J. Smoller, Solutions in the large for some nonlinear hyperbolic conservation laws, \emph{Comm. Pure Appl. Math.} \textbf{26} (1973), 183--200.

\bibitem{PEP1}
Y.-J. Peng, Uniformly global smooth solutions and convergence of Euler--Poisson systems with small parameters, \emph{SIAM J. Math. Anal.} \textbf{47} (2015), no. 2, 1355--1376.

\bibitem{pengJFA}
Y.-J. Peng, Global large smooth solutions for isothermal Euler equations with damping and small parameter, \emph{J. Funct. Anal.} \textbf{287} (2024), Paper No. 110571.

\bibitem{13}
Y.-J. Peng and V. Wasiolek, Uniform global existence and parabolic limit for partially dissipative hyperbolic systems, \emph{J. Differential Equations} \textbf{260} (2016), no. 9, 7059--7092.

\bibitem{P}
B. Perthame, Non-existence of global solutions to Euler--Poisson equations for repulsive forces, \emph{Japan J. Appl. Math.} \textbf{7} (1990), 363--367.

\bibitem{prv}
F. Poupaud, M. Rascle, and J.-P. Vila, Global solutions to the isothermal Euler--Poisson system with arbitrarily large data, \emph{J. Differential Equations} \textbf{123} (1995), no. 1, 93--121.


\bibitem{SK} Y. Shizuta and S. Kawashima, Systems of equations of hyperbolic-parabolic type with applications to the discrete Boltzmann equation, \emph{Hokkaido Math. J.} \textbf{14} (1985), 249--275.

\bibitem{16}
T. C. Sideris, Formation of singularities in three-dimensional compressible fluids, \emph{Comm. Math. Phys.} \textbf{101} (1985), 475--485.


\bibitem{stein}
E. M. Stein, \emph{Singular Integrals and Differentiability Properties of Functions}, Princeton Mathematical Series, Vol. 30, Princeton University Press, Princeton, 1970.



\bibitem{TW1}
E. Tadmor and D. Wei, On the global regularity of sub-critical Euler--Poisson equations with pressure, \emph{J. Eur. Math. Soc.} \textbf{10} (2008), 757--769.

\bibitem{Taylor1996}
M. E. Taylor, \emph{Partial Differential Equations III: Nonlinear Equations}, Appl. Math. Sci., vol. 117, Springer, New York, 1996.

\bibitem{WTB}
D. Wei, E. Tadmor, and H. Bae, Critical thresholds in multi-dimensional Euler--Poisson equations with radial symmetry, \emph{Commun. Math. Sci.} \textbf{10} (2012), no. 1, 75--86.

\bibitem{W1}
G. B. Whitham, \emph{Linear and Nonlinear Waves}, Wiley, New York, 1974.

\bibitem{XuSIAM} J. Xu, Relaxation-time limit in the isothermal hydrodynamic model for semiconductors, \emph{SIAM J. Math. Anal.} (2009), \textbf{40}, no. 5, 1979-1991.

\bibitem{XuWang}
J. Xu and Z. Wang, Relaxation limit in Besov spaces for compressible Euler equations, \emph{J. Math. Pures Appl. (9)} \textbf{99} (2013), 43--61.

\bibitem{Y1}
W.-A. Yong, Diffusive relaxation limit of multidimensional isentropic hydrodynamical models for semiconductors, \emph{SIAM J. Appl. Math.} \textbf{64} (2004), no. 5, 1737--1748.

\end{thebibliography}
\end{document}